\documentclass{article}

\usepackage{calc}
\usepackage{array}
\usepackage{graphicx}
\usepackage{subfigure}
\usepackage[indent,bf]{caption}
\usepackage{rotating}
\usepackage{setspace}
\usepackage{longtable}
\usepackage{mdframed}
\usepackage{rotating}
\usepackage{amssymb,amsthm}
\usepackage{amsmath, amsfonts, epsfig}
\usepackage{epsfig,latexsym}
\usepackage{multicol, paralist}
\usepackage{pgf,tikz}
\usetikzlibrary{arrows,snakes,backgrounds, automata}
\usepackage{verbatim}
\usepackage{url}
\usepackage{fancyhdr}
\usepackage{comment}
\usepackage{xcolor}
\usepackage{authblk}

\usetikzlibrary{arrows}

\newtheorem{thm}{Theorem}[section]
\newtheorem{prob}[thm]{Problem}

\newtheorem{lem}[thm]{Lemma}
\newtheorem{prop}[thm]{Proposition}
\newtheorem{cor}[thm]{Corollary}

\newcommand{\R}{\mathbb{R}}
\newcommand{\Z}{\mathbb{Z}}
\newcommand{\qedclaim}{\hfill $\diamond$ \medskip}
\newenvironment{proofclaim}{\noindent{\em Proof.}}{\qedclaim} 

\renewcommand*{\thefootnote}{\fnsymbol{footnote}}

\usepackage{fancyhdr}
\DeclareMathOperator{\Cay}{Cay}
\DeclareMathOperator{\fed}{\gamma_{\,\textit{f}}^{\infty}}
\DeclareMathOperator{\med}{\gamma_{\textrm{m}}^{\infty}}

\pgfdeclarelayer{edgelayer}
\pgfdeclarelayer{nodelayer}
\pgfsetlayers{edgelayer,nodelayer,main}

\tikzstyle{none}=[inner sep=0pt]
\definecolor{hexcolor0xf81e1c}{rgb}{0.973,0.118,0.110}
\definecolor{hexcolor0x3c00ff}{rgb}{0.235,0.000,1.000}

\tikzstyle{whitevertex}=[circle,fill=white,draw=black, scale = 0.7]
\tikzstyle{redvertex}=[circle,fill=hexcolor0xf81e1c,draw=black, scale = 0.7]
\tikzstyle{bluevertex}=[circle,fill=hexcolor0x3c00ff,draw=black, scale = 0.7]
\tikzstyle{greenvertex}=[circle,fill=green,draw=black, scale=0.7]
\tikzstyle{purplevertex}=[circle,fill=magenta,draw=black, scale=0.7]
\tikzstyle{grayvertex}=[circle,fill=white,draw=gray, scale=0.7]
\tikzstyle{blackvertex}=[circle,fill=black,draw=black, scale=0.7]

\tikzstyle{textbox}=[rectangle,fill=none,draw=none]
\tikzstyle{box}=[rectangle,fill=none,draw=black]

\tikzstyle{arc}=[black, ->]
\tikzstyle{grayarc}=[gray, ->]
\tikzstyle{bluearc}=[blue, ->]
\tikzstyle{grayedge}=[draw=gray]
\tikzstyle{blueedge}=[draw=blue]
\tikzstyle{rededge}=[draw=red]
\tikzstyle{edge}=[draw=black]

\tikzstyle{vertex}=[circle, ,fill=white,draw=black, scale=0.66]

\tikzstyle{10circle}=[circle, scale=10.0,draw=black] 
\tikzstyle{10oval}=[ellipse, scale=10.0,draw=black]

\author[1]{%
Fnu Devvrit}
\affil[1]{\footnotesize Department of Computer Science,
University of Texas at Austin,
Austin, TX, U.S.A.}

\author[2]{%
Aaron Krim-Yee}
\affil[2]{%
Department of Bioengineering,
McGill University,
Montreal, QC, Canada}

\author[3]{%
Nithish Kumar}
\affil[3]{Department of Computer Science,
Purdue University,
West Lafayette, IN, USA}

\author[4]{%
Gary MacGillivray}
\affil[4]{Department of Mathematics and Statistics,
University of Victoria,
Victoria, BC, Canada}

\author[5,6]{%
Ben Seamone}
\affil[5]{Mathematics Department, Dawson College, Montreal, QC, Canada}
\affil[6]{DIRO, Universit\'e de Montr\'eal, Montreal, QC, Canada}

\author[4]{%
Virg\'elot Virgile}

\author[7]{%
AnQi Xu}
\affil[7]{D\'epartement de m\'edecine,
Universit\'e de Montr\'eal,
Montreal, QC, Canada \\}

\newcommand\extrafootertext[1]{%
    \bgroup
    \renewcommand\thefootnote{\fnsymbol{footnote}}%
    \renewcommand\thempfootnote{\fnsymbol{mpfootnote}}%
    \footnotetext[0]{#1}%
    \egroup
}

\title{Fractional eternal domination: securely distributing resources across a network}

\begin{document}

\maketitle

\begin{abstract}
This paper initiates the study of fractional eternal domination in graphs, a natural relaxation of the well-studied eternal domination problem.  We study the connections to flows and linear programming in order to obtain results on the complexity of determining the fractional eternal domination number of a graph $G$, which we denote $\fed(G)$.   We study the behaviour of $\fed(G)$ as it relates to other domination parameters.  We also determine bounds on, and in some cases exact values for, $\fed(G)$ when $G$ is a member of one of a variety of important graph classes, including trees, split graphs, strongly chordal graphs, Kneser graphs, abelian Cayley graphs, and graph products.
\end{abstract}

\section{Introduction}

\extrafootertext{\noindent Email addresses: devvrit@cs.utexas.edu, aaron.krim-yee@mail.mcgill.ca, kumar410@purdue.edu, gmacgill@uvic.ca, bseamone@dawsoncollege.qc.ca, virgilev@uvic.ca, an.qi.xu@umontreal.ca}

Let $G = (V,E)$ be a graph.  We denote by $N_G(u)$ the \textit{(open) neighbourhood} of $u \in V(G)$, or the set of vertices which are adjacent to $u$ in $G$ (one may write $N(u)$ if $G$ is clear from context).  The \textit{closed neighbourhood} of $u$ is $N_G[u] = N_G(u) \cup \{u\}$.  The closed neighbourhood of a set $X \subseteq V(G)$ is $N_G[X] = \cup_{u \in X} N_G[u]$ (the open neighbourhood may be defined similarly).  A set $X \subseteq V(G)$ is called a \textit{dominating set} if $N_G[X] = V(G)$.
The cardinality of a minimum dominating set in $G$ is denoted $\gamma(G)$; this parameter is called the \textit{domination number} of $G$.  A well-studied variation of the domination number of graphs is the so-called fractional domination number (see, e.g., \cite{haynes1998domination, haynes1998fundamentals}).
A \emph{fractional dominating function} of $G$ is a function $w \colon V(G) \to \R$ such that $\sum_{x \in N[v]} w(x) \geq 1$ for all $v \in V$.
The \emph{total weight} of $w$ is $\sum_{x \in V} w(x)$.
A graph $G$ is \textit{$S$-fractionally dominated} if there exists a fractional dominating function of $G$ with total weight less than or equal to $S$.
The \textit{fractional domination number} of  $G$ is the smallest total weight of a fractional
dominating function of $G$; the parameter is denoted $\gamma_f(G)$.  

Many recent papers have considered dynamic models of graph ``protection'', where agents move through a graph in a way that somehow responds to ``attacks''.  We refer the reader to \cite{Klostermeyer2016protect} for a survey of models related to graph domination which includes both known results and many interesting conjectures.  Our work in this paper follows a line of research which originates from \cite{BCGMVW04}, where the ``eternal domination'' model was introduced.  We describe this model in terms of a two-player game, played between \textit{defender} and \textit{attacker}.  The defender controls a set of guards which occupy some subset of $V(G)$ (typically, only one guard is allowed to occupy any one vertex).  The attacker will attack some vertex in the graph, which forces the defender to respond to that attack.  More precisely, the defender chooses some set $D_1 \subseteq V(G)$ as the starting positions for the guards, and will choose each subsequent set $D_{i+1}$, $i \geq 1$, in response to the game-play of the attacker in the $i$-th round (this is sometimes referred to as the \textit{adaptive online model} of the game).  For each $i \geq 1$, the attacker's move in the $i$-th round is to choose some vertex $v_i \notin D_i$.  The defender must then choose some vertex $u_i \in D_i$ such that $v_i \in N(u_i)$, and set $D_{i+1} = D_i \cup \{v_i\} \setminus \{u_i\}$.  The goal of the defender is to be able to respond to any infinite sequence of attacks.  If the defender can win from some set $D_1$, then $D_1$ is called an eternal dominating set; note that such a set $D_1$ (and each subsequent $D_i$) must necessarily be a dominating set.  The \textit{eternal domination number} of $G$, denoted $\gamma^{\infty}(G)$, is the minimum cardinality of an eternal dominating set in $G$.  Recall that a \textit{clique} in $G$ is a subset of $V(G)$ whose elements are pairwise adjacent, and an \textit{independent set} (or \textit{stable set}) is a subset of $V(G)$ whose elements are pairwise non-adjacent.  The \textit{clique cover number} of $G$, denoted $\theta(G)$, is the minimum cardinality of a collection of cliques of $G$ whose union is $V(G)$.  The \textit{independence number} of $G$, denoted $\alpha(G)$, is the maximum cardinality of an independent set of $G$.  It is easy to argue (see \cite{goddard2005eternal}), that $$\alpha(G) \leq \gamma^{\infty}(G) \leq \theta(G).$$
One may consider a related model where, instead of only moving one guard to respond to an attack, one may reconfigure $D_i$ to $D_{i+1}$ by moving any number of guards between adjacent vertices so long as the attacked vertex receives one guard.  This is called the m\textit{-eternal domination model} (introduced in \cite{goddard2005eternal}, and an initial set of vertices that can guard any sequence of attacks is called an m-\textit{-eternal dominating set}.  The minimum cardinality of an m-eternal dominating set, denoted $\med(G)$, is the m-textit{-eternal domination number} of $G$.  It is clear that $\gamma(G) \leq \med(G)$.  By a clever application of Hall's Theorem, given in \cite{goddard2005eternal}, it has been shown that $\med(G) \leq \alpha(G)$; thus we have the following fundamental inequality chain:
$$\gamma(G) \leq \med(G) \leq \alpha(G) \leq \gamma^{\infty}(G) \leq \theta(G).$$

We consider an eternal domination model that may be considered as the fractional relaxation of m-eternal domination, which we call \textit{fractional eternal domination}\footnote{This model could also be called \textit{fractional {\rm m}-eternal domination}, allowing for the possibility of a fractional relaxation of the one-guard move model.  We suppress the ``m-'' throughout for the sake of simplicity.}.  We assign non-negative real weights to $V(G)$ so that $S$-fractional domination is maintained for some fixed value $S$ subject to vertex attacks.  Denote the weight at vertex $v$ at time-step $i$ by $w_i(v)$, and write $w(v)$ for the initial weight $w_1(v)$ of the vertex $v$.  After the $i$-th attack at $v_i$, the defender may move weight from any vertex $x$ to the vertices in $N(x)$.  If $m_{xy,i}$ denotes the weight moved from $x$ to $y$ in round $i$, then we require only that $\sum_{y \in N(x)}m_{xy,i} \leq w_{i}(x)$.  The defender may do this simultaneously for as many vertices as necessary, but the resulting weight function $w_{i+1}$ must $S$-fractionally dominate the graph and $w_{i+1}(v_i) \geq 1$.  We denote by $\fed(G)$ the infimum over all $S$ for which $G$ can be eternally $S$-fractionally dominated, and call this the \textit{fractional eternal domination number} of $G$.
Note that if one restricts all quantities in the above description to be integral, then one recovers the $m$-eternal domination model.  

In Section \ref{sec:linprog}, we look at fractional eternal domination through the lens of linear programming.  Section \ref{sec:behaviour} examines some basic properties of $\fed(G)$, in particular as it relates to other domination parameters.  In Sections \ref{sec:basic}, \ref{sec:Cayley}, and \ref{sec:products} we establish properties of $\fed(G)$ when $G$ is a member of a number of important graph classes.

\section{Linear programming and reconfiguration}\label{sec:linprog}

We begin with a look at the fractional eternal domination problem through a linear programming lens, and give conditions under which a graph's fractional eternal domination number, as well as a guarding strategy, can be computed efficiently.

Let $G$ be a graph with $V(G) = \{1, 2, \ldots, n\}$, and assign to vertex $i$ a variable $x_i$. Denote by $N_i = N[i]$, and for $S \subseteq V(G)$ let $N_S = \cup_{i \in S} N_i$.  Without loss of generality, we may assume that vertex $1$ is attacked first.  The quantity $\fed(G)$ must then be at least the solution to the following LP, which corresponds to the minimum weight of a fractional dominating set in which vertex 1 has weight at least 1: \newline

\begin{compactitem}
    \item[\textbf{Minimize}] $\sum \limits_{i=1}^n x_i$ subject to
    \item $x_1 \geq 1$
    \item $x_i \geq 0$ for all $i = 2,\ldots,n$
    \item $\sum \limits_{j \in N_i} x_j \geq 1$ for all $i = 2,\ldots,n$. \\
\end{compactitem}

Let $w_1$ and $w_2$ be fractional dominating functions of $G$
with the same total weight.
We say that \emph{$w_1$ can be reconfigured to $w_2$} if, for $1 \leq i \leq n$, the weight $w(i)$ can
be redistributed to vertices in $N[i]$ so that the resulting fractional dominating function is $w_2$.

Let $N_{w_1, w_2}$ be the network with vertex set $V(N_{w_1, w_2}) = \{s, t\} \cup \{1, 2, \ldots, n\} \cup \{1',2', \ldots, n'\}$ and
 arc set $A(N_{w_1, w_2}) = \{s i: 1 \leq i \leq n\} \cup \{i j': j \in N_G[i]\} \cup \{i' t: 1 \leq i' \leq n\}$.
 The capacity of the arc $si$ is $w_1(i),\ 1 \leq i \leq n$.
 The capacity of the arc $i't$ is $w_2(i'),\ 1 \leq i' \leq n$.
 Each arc in the set $\{i j': j \in N_G[i]\}$ has infinite capacity.

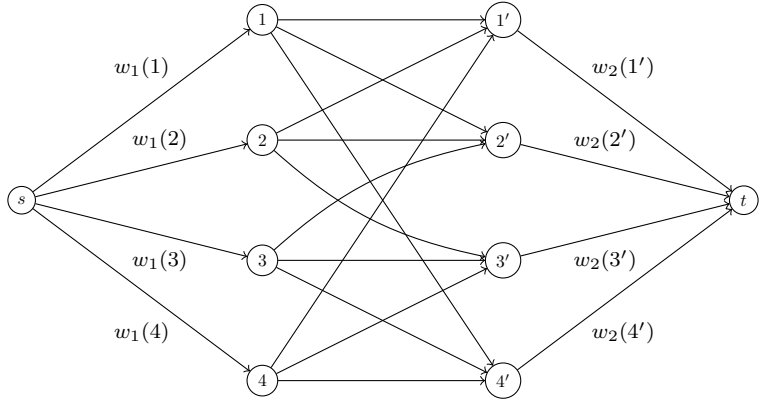
\begin{figure}[h!]
    \centering
	\begin{tikzpicture}[scale=.8]
		\begin{pgfonlayer}{nodelayer}
			\node [style=vertex] (0) at (-6, 0) {$s$};
			\node [style=vertex] (1) at (-2, 3) {$1$};
			\node [style=vertex] (2) at (-2, 1) {$2$};
			\node [style=vertex] (3) at (-2, -1) {$3$};
			\node [style=vertex] (4) at (-2, -3) {$4$};
			\node [style=vertex] (5) at (2, 3) {$1'$};
			\node [style=vertex] (6) at (2, 1) {$2'$};
			\node [style=vertex] (7) at (2, -1) {$3'$};
			\node [style=vertex] (8) at (2, -3) {$4'$};
			\node [style=vertex] (9) at (6, 0) {$t$};
		\end{pgfonlayer}
		\begin{pgfonlayer}{edgelayer}
			\draw [style=arc] (0) to (1);
			\node[] at (-4,2.2) {\footnotesize $w_1(1)$};
			\draw [style=arc] (0) to (2);
			\node[] at (-3.7,1) {\footnotesize $w_1(2)$};
			\draw [style=arc] (0) to (3);
			\node[] at (-3.7,-1) {\footnotesize $w_1(3)$};
			\draw [style=arc] (0) to (4);
			\node[] at (-4,-2.2) {\footnotesize $w_1(4)$};
			\draw [style=arc] (1) to (5);
			\draw [style=arc] (1) to (6);
			\draw [style=arc] (1) to (8);
			\draw [style=arc] (2) to (6);
			\draw [style=arc] (2) to (5);
			\draw [style=arc] [bend right=15] (2) to (7);
			\draw [style=arc] (3) to (7);
			\draw [style=arc] [bend left=15] (3) to (6);
			\draw [style=arc] (3) to (8);
			\draw [style=arc] (4) to (8);
			\draw [style=arc] (4) to (7);
			\draw [style=arc] (4) to (5);
			\draw [style=arc] (5) to (9);
			\node[] at (4,2.2) {\footnotesize $w_2(1')$};
			\draw [style=arc] (6) to (9);
			\node[] at (3.7,1) {\footnotesize $w_2(2')$};
			\draw [style=arc] (7) to (9);
			\node[] at (3.7,-1) {\footnotesize $w_2(3')$};
			\draw [style=arc] (8) to (9);
			\node[] at (4,-2.2) {\footnotesize $w_2(4')$};
		\end{pgfonlayer}
	\end{tikzpicture}
	
	\label{Reconfiguration}
	\caption{Reconfiguration network $N_{w_1,w_2}$ for $C_4$, with $V = \{1,2,3,4\}, E = \{12,23,34,14\}$}
\end{figure}

 The following lemma follows easily from the definition of a flow.
 
\begin{lem}
Let $G$ be a graph with $V(G) = \{1, 2, \ldots, n\}$.
Let $w_1$ and $w_2$ be fractional dominating functions of $G$
with the same total weight.
Then $w_1$ can be reconfigured to $w_2$ if and only if there is a flow from $s$ to $t$ in $N_{w_1, w_2}$ with value 
$\sum\limits_{i = 1}^n w_1(i)$.
 \end{lem}

In the fractional eternal domination problem, 
defending $G$ from an infinite sequence of attacks requires a collection of fractional dominating functions of the same total weight such that:
\begin{enumerate}
\item  for each $i, 1 \leq i \leq n$ there exists at least one function in the collection in which the weight assigned to vertex $i$ is at least 1, and 
\item  for each fractional dominating function $w$ in the collection and  each $j, 1 \leq j \leq n$, there exists a fractional dominating function $w_j$ in the collection in which the weight assigned to vertex $j$ is at least 1 and $w$ can be reconfigured to $w_j$.
\end{enumerate}
It follows that there are $n$ fractional dominating functions
$w_1, w_2, \ldots, w_n$ (not necessarily distinct) of the same total weight and such that $w(i) = 1,\ 1 \leq i \leq n$.
If we impose the additional condition that each of these can be reconfigured to every other function in the collection, then the smallest total weight for which there exists such a collection of fractional dominating functions is an upper bound on the fractional eternal domination number.
We show that this quantity can be determined as the solution to a linear program with rational constraints, and hence is rational.

Let $G$ be a graph with $V  = \{1, 2, \ldots, n\}$, and
let $w$ be a fractional dominating function of $G$.
Then $w$ corresponds to an $n$-tuple $X = (x_1, x_2, \ldots, x_n)$ where $x_i = w(i)$.
In this work, we will use the functional notation and the $n$-tuple notation interchangeably and will refer to $X$ as
being a fractional dominating function.

For $1 \leq i \leq n$, 
$X_i = (x_{i 1}, x_{i 2}, \ldots, x_{i n})$, and let
$\mbox{FDS}_i$ denote the collection of inequalities corresponding to 
$X_i$ being a fractional dominating function:
$$\mbox{FDS}_i: \quad x_{ii} = 1, \mbox{ and } \sum_{j \in N_G[k]} x_{ij} \geq 1, \quad 1 \leq k \leq n.$$

For $1 \leq i, j \leq n$ and $i \neq j$, let $\mbox{RECONFIG}_{ij}$ denote the set of constraints in the linear program corresponding to determining the maximum value of a flow from $s$ to $t$ in $N_{X_i, X_j}$, together with the constraint that the value of the flow is $\sum\limits_{\ell = 1}^n x_{i \ell}$.

Based on the discussion above, an optimal solution to the following linear program $A$ provides an upper bound on the fractional eternal domination number and the collection $X_1, X_2, \ldots, X_n$ of fractional dominating functions of the given total weight that can be used to defend $G$.

\bigskip
\noindent $A:\ \ $ {\bf Minimize} $\sum\limits_{j=1}^n x_{1 j}$ subject to
\begin{compactitem}
\item  FDS$_i, \quad  i = 1, \ldots, n$
\item  $\sum\limits_{j=1}^n x_{i, j} = \sum\limits_{j=1}^n x_{1, j}, \quad i = 2, \ldots, n$
\item RECONFIG$_{ij}, \quad   i = 1, \ldots, n,\ \ j = 1, \ldots, n,  \mbox{ and } i \neq j$
\item $x_{ij} \geq 0, \quad  i = 1, \ldots, n,\ \ j = 1, \ldots, n$
\end{compactitem}

\bigskip
If the optimal solution to linear program $A$ equals $\gamma_f^\infty$, then
we say \emph{$G$ can be eternally fractionally dominated by $n$ f.d.-functions}.

Since the number of constraints in the linear program $A$ is polynomial in $n$, and
linear programming problems are solvable in polynomial time, we have the following:

\begin{prop}
If $G$ can be eternally fractionally dominated by $n$ f.d.-functions, then $\fed(G)$
and a strategy for eternally fractionally dominating $G$ can be computed in polynomial time.
Further, $\fed(G)$ is a rational number.
\label{nsets}
\end{prop}

We will show later that Proposition \ref{nsets} applies to split graphs.

\section{Behaviour of $\fed(G)$}\label{sec:behaviour}

In this section, we compare $\fed(G)$ to some of the domination parameters given in the introduction and explore which numerical values of $\fed(G)$ are possible. We begin with some obvious propositions.

\begin{prop}\label{prop:basicbound}
For any graph $G$, $\gamma_f(G) \leq \fed(G) \leq \med(G)$.
\end{prop}

To see another upper bound on $\fed(G)$, note that doubling the weight of every vertex in a minimal fractional dominating function on $V(G)$ gives a straightforward guarding strategy -- let $w_1$ be a fractional dominating function of weight $\gamma_f(G)$ and initially let $w_2 = w_1$.  After an attack in an odd-numbered time step, weights from $w_1$ are moved to respond, while weights from $w_2$ are returned to their initial assignments.  Similarly, after an attack in an even-numbered time step, weights from $w_2$ are moved to respond, while weights from $w_1$ are returned to their initial assignments.  

\begin{prop}\label{prop:fracupperbound}
For any graph $G$, $\fed(G) \leq 2\gamma_f(G)$.
\end{prop}

Note that the bound in Proposition \ref{prop:fracupperbound} is tight, in that for every $k \geq 1$ there exists a graph $G$ for which $\fed(G) = 2k$ and $\gamma_f(G) = k$.  Let $G$ be a graph obtained from a path $P = v_1v_2\cdots v_{3k}$ by adding two leaves $\{x_i,y_i\}$ to $v_i$ for each $i \equiv 2 \pmod 3$.  It is easy to check that $\gamma_f(G) = k$.  Attacking $x_{3r+2}$ then $y_{3r+2}$ requires a total weight of at least $2$ to be present in $N_G[v_{3r+2}]$, and thus it follows that $\fed(G) \geq 2k$.

From the definition of $\fed(G)$, it may not be immediately apparent that $\fed(G)$ can be non-integral, or that it can even differ from the bounds in Proposition \ref{prop:basicbound}.  We show how to construct infinite families of graphs with non-integral values of $\fed(G)$ which also differ from the bounds of Proposition \ref{prop:basicbound}.

Before proceeding, we note the following easy observation.

\begin{prop}\label{prop:noncomplete}
If $G$ is not a complete graph, then $\fed(G) \geq 2$.
\end{prop}

\begin{proof}
Let $x,y \in V(G)$ be nonadjacent.  If $x$ is attacked, then $w(x) \geq 1$.  For $y$ to be guarded, $w(N[y]) \geq 1$.  Since $x\notin N[y]$, the result follows.
\end{proof}

Our construction shows that any admissible rational value of $\fed(G)$ (that is, equal to $1$ in the case of complete graphs or at least $2$ in the case of non-complete graphs) is possible.

\begin{thm}\label{Gtd}
For any rational number $q > 2$, there exists a graph $G$ such that $\gamma_f(G) < \fed(G) < \gamma_{\rm m}^{\infty}(G)$ and $\fed(G) = q$.
\end{thm}

\begin{proof}
For two positive integers $t \geq d$, let $X = [t]$ and $Y = {[t] \choose d}$. Let $Y'$ be a set of cardinality ${t \choose d}$ which contains a copy $y'$ of each element $y \in Y$.  
Denote by $\mathcal{G}_{t,d}$ the graph with $V = X \cup Y \cup Y'$ obtained by taking a complete graph on $X$ and an edge between $x \in X$ and $y \in Y$ and an edge between $x \in X$ and $y' \in Y'$ if and only if $x \in y$. We first prove that $\med(\mathcal{G}_{t,d}) = t-d+2$.  It is shown in \cite{BDEMY17} that a split graph $G$ satisfies $\med(G) \in \{\gamma(G),\gamma(G)+1\}$, and $\med(G) = \gamma(G)$ if and only if every vertex in the independent set is domination-critical. Since $t-d+1$ elements from $t$-set are necessary and sufficient to hit every $d$-subset, we have that $\gamma(\mathcal{G}_{t,d}) = t-d+1$.  Since $Y \cup Y'$ is the independent set and no vertex in this set is domination-critical (each vertex has a ``twin'', or a vertex with the same neighbourhood), it follows that $\med(G) = \gamma(G)+1 = t-d+2$.

We now claim that $\fed(\mathcal{G}_{t,d}) = 1 + \frac{t}{d}$.  We allow two possible types of fractional eternal dominating functions.  In both, a weight of $\frac{1}{d}$ is assigned to every vertex in $X$.  We may then either have an additional weight of $1$ added to a vertex in $X$ (state 1) or a weight of $1$ on some vertex in $Y \cup Y'$ (state 2).  We let $z$ denote the vertex with weight at least $1$ in the argument below.

Suppose we are in state $1$ and an attack happens at $x \in X$.  A weight of one is easily passed from $z$ to $x$ as they are adjacent.  If $y \in Y \cup Y'$ is attacked, each of its neighbours sends weight $\frac{1}{d}$ and $z$ sends $\frac{1}{d}$ to each neighbour $y$.  
If we are in state $2$, then an attack in $X$ is handled similarly to the previous case.  If an attack happens at $y \in Y \cup Y'$, then $z$ sends $\frac{1}{d}$ to each of its neighbours, vertices in $N(z) \setminus N(y)$ send weight $\frac{1}{d}$ to vertices in $N(y) \setminus N(z)$ (this is easy as all vertices are in $X$), and all vertices in $N(y)$ send $\frac{1}{d}$ to $y$.  In all cases, we finish in state 1 or state 2.

To see that no weighting with lower total weight is possible, we simply note that any attempt to lower the weight of one vertex in $X$ would require increasing the weights of the other vertices in $X$ by at least that amount to maintain domination of $Y$, and that it is never necessary to have weight assigned to vertices in $Y$ since their closed neighbourhoods are contained in the closed neighbourhoods of vertices in $X$.  Thus
	\begin{compactitem}
    \item $\gamma_f(\mathcal{G}_{t,d}) = \frac{t}{d}$,
    \item $\fed(\mathcal{G}_{t,d}) = 1 + \frac{t}{d}$, and
    \item $\gamma_{\rm m}^\infty(\mathcal{G}_{t,d}) = \gamma(\mathcal{G}_{t,d})+1 = t-d+2$. \qedhere
    \end{compactitem}
\end{proof}
    

Having proven that $\fed(G)$ can take on any admissible rational value, we now note that the construction of $\mathcal{G}_{t,d}$ can be easily modified to provide an infinite family of graphs for which $\fed(G) = q$ for any admissible rational value of $q$.

\begin{thm}\label{Gmn}
For any rational number $q \geq 2$, there exists an infinite family of graphs $\mathcal{G}_q$ such that $\fed(G) = q$ for every $G \in \mathcal{G}_q$.
\end{thm}

\begin{proof}
Let $q = t/d + 1$ and suppose that $V(\mathcal{G}_{t,d})$ partitions into a clique $X$ and independent set $Y$.  Let $\mathcal{G}_q$ be the family of graphs obtained from $\mathcal{G}_{t,d}$ by replacing each $y \in Y$ with any arbitrary graph $H_y$ and joining each vertex of $H_y$ to $N_{\mathcal{G}_{t,d}}(y)$.  By applying the same argument as in the proof of Theorem \ref{Gtd}, it follows that for any $G \in \mathcal{G}_q$,  $\fed(G) = \fed(\mathcal{G}_{t,d}) = 1+\frac{m}{n} = q$.
\end{proof}

In light of our linear programming discussion from Section \ref{sec:linprog}, we also give a linear programming lower bound for $\fed(G)$.  For a given graph $G$, denote by $f(v)$ the least total weight of a fractional dominating function $D_v$ in which the
weight assigned to vertex $v$ is at least 1 (note that $f(v)$ can be computed in polynomial time by linear programming).
Let $F(G) = \max\{f(v): v \in V\}$.
\begin{prop}
For any graph $G$, $\fed(G) \geq F(G)$.
\end{prop}

We now turn our attention to a comparison of $\fed(G)$ and $\gamma(G)$, each of which is bounded below by $\gamma_f(G)$ and bounded above by $\med(G)$.  We will see later that, for certain graph classes, $\fed(G) \geq \gamma(G)$ (see Corollaries \ref{packcor1} and \ref{packcor2}); however, we show here that the two parameters are, in general, not comparable.

Recall that, by Proposition \ref{prop:fracupperbound}, we have that $\fed(G) \leq 2\gamma_f(G)$ for any graph $G$, and so $\fed(G) \leq 2\gamma(G)$.  In other words, $\fed(G)$ cannot grow unboundedly large in terms of $\gamma(G)$.  On the other hand, the construction of $\mathcal{G}_{t,d}$ in Theorem \ref{Gtd} and its modification in Theorem \ref{Gmn} gives us the following:

\begin{cor}\label{unbounded1}
For any $\varepsilon>0$ and any rational number $q \geq 2$, there exists an infinite family of graphs $\mathcal{G}_q$ such that $q = \fed(G) < \varepsilon \gamma(G)$ for each $G \in \mathcal{G}_q$.
\end{cor}

\begin{proof}
Let $\alpha > \frac{1}{\varepsilon}$ be some sufficiently large integer, $q = \frac{t}{d}+1$ where $\frac{t}{d}$ is in reduced form,  $t' = \alpha t$, and $d' = \alpha d$.
 Consider the family of graphs $\mathcal{G}_q$ as constructed in Theorem \ref{Gmn}, with parameters $t'$ and $d'$ in place of $t$ and $d$, respectively.  Let $G \in \mathcal{G}_q$.  We have that $\fed(G) = \frac{t'}{d'}+1= \frac{t}{d}+1$ and $\gamma(G) = t'-d'+1 = \alpha(t-d)+1$.  If $\alpha$ is sufficiently large, then $\frac{t}{d}+1 < \varepsilon\left(\alpha(t-d)+1\right)$ as desired.
\end{proof}

Corollary \ref{unbounded1} also immediately implies that $\fed(G)$ may be an arbitrarily small fraction of $\med(G)$.  Furthermore, the construction of $\mathcal{G}_{m,n}$ in Theorem \ref{Gmn} shows that for any graph $H$ there exists a graph $G$ containing $H$ as an induced subgraph for which $\fed(G) < \varepsilon \gamma(G)$.

Finally, though the ratio $\fed/\gamma$ cannot be made unboundedly large, the difference $\fed(G) - \gamma(G)$ can be made arbitrarily large -- in the next section we show that $\fed(P_n) = \lceil\frac{n}{2}\rceil$, whereas $\gamma(P_n) = \lceil\frac{n}{3}\rceil$ for any $n \in \mathbb{N}$.


\section{Basic graph classes}\label{sec:basic}

For some specific classes of graphs, exact values of $\fed(G)$ are easy to compute (proofs are left to the reader).  To guard a complete graph, a weight of $1$ placed on any vertex is necessary and sufficient.  To guard a path, consider a sequence of attacks on a maximum independent set ``from left to right'' to see that weight $1$ on alternating vertices is necessary and sufficient.  For cycles, it is known (see, e.g., \cite{goddard2005eternal}) that $\med(C_n) = \lceil \frac{n}{3} \rceil$ and so $\fed(C_n) \leq \lceil \frac{n}{3} \rceil$.  To see that $\fed(C_n) \geq \lceil \frac{n}{3} \rceil$, note that total weight less than $\lceil \frac{n}{3} \rceil$ is not enough weight to respond to attacks on every third vertex, in order, around the cycle. 

\begin{prop}\label{prop:basicgraphs}
For any $n \in \mathbb{N}$:
	\begin{enumerate}
    	\item $\fed(K_n)=1$.
		\item $\fed(P_n)=\alpha(P_n)=\lceil \frac{n}{2} \rceil$.
        \item $\fed(C_n)=\gamma(C_n)=\lceil \frac{n}{3} \rceil$ if $n \geq 3$.
	\end{enumerate}
\end{prop}

\begin{prop}
If $G$ has a universal vertex or if every edge is a dominating edge of $G$, then  $\gamma^{\infty}_m(G) \leq 2$ and thus $\fed(G) \leq 2$.
\end{prop}

Since every edge in a complete multipartite graph is dominating, the following is easily obtained.

\begin{cor}
For all positive integers $n_1, n_2, \ldots, n_k$ where $n_k \geq 2$, $\fed(K_{n_1, n_2, \ldots, n_k})= 2$.
\end{cor}

Finally, we give the exact value of $\fed(G)$ when $G$ is a tree, noting that a linear time algorithm is given in \cite{KM09} which computes $\med(T)$ for any tree $T$.

\begin{prop}
	For any tree $T$, $\fed(T)=\med(T)$.
\end{prop}

\begin{proof}
	We prove this by induction on $n$, the order of the tree. The cases $n=1, 2$ are clear. Let $T$ be a tree of order $n>2$ and suppose the proposition is true for all trees of order strictly less than $n$. If $T$ is a star, then the result follows. Suppose $T$ is not a star. From Proposition \ref{prop:basicbound} we see that it is sufficient to show $\fed(T) \geq \gamma_{\rm m}^\infty(T)$. Recall that the eccentricity of a vertex is the maximum distance from that vertex to all others.  Since $T$ is not a star, $T$ contains a vertex of eccentricity at least $2$ which is adjacent to at least one leaf.  Let $x$ be such a vertex with the largest eccentricity. If $x$ is adjacent to exactly one leaf $y$, then a weight summing to at least $1$ must be maintained on $x$ and $y$ at all time and the tree $T'$ obtained by deleting $x$ and $y$ satisfies $\gamma_f^\infty(T')=\gamma_{\rm m}^\infty(T')=\gamma_{\rm m}^\infty(T)-1$. If $x$ is adjacent to at least two leaves, then a weight summing to at least $2$ must be maintained in the closed neighbourhood of $x$ at all time and the tree $T'$ obtained by deleting these leaves satisfies $\gamma_f^\infty(T')=\gamma_{\rm m}^\infty(T')=\gamma_{\rm m}^\infty(T)-1$. In both cases, the result follows.
\end{proof}

\subsection{Split Graphs}

Recall that a graph is called a \textit{split graph} if its vertex set can be partitioned into a clique and an independent set. 

\begin{thm}
If $G$ is a split graph, then $\gamma_f(G) \leq \fed(G) \leq 1 + \gamma_f(G)$.
\end{thm}

\begin{proof}
The lower bound is trivial (Proposition \ref{prop:basicbound}), and so we need only to prove the upper bound.  Let $X \cup Y$ be a partition of $V(G)$ for which $X$ is a clique and $Y$ is an independent set.  Begin with a fractional dominating function on $G$, say $w$; note that this is precisely a solution to $\left\{\sum_{u \in N(y)} w(u) \geq 1 \,\mid\, y \in Y\right\}$. Assign an additional weight $1$ to some arbitrary vertex.  We show that, after any attack, we can maintain the fractional dominating function $w$ with some arbitrary vertex receiving an additional weight of $1$; $z$ denotes this special vertex throughout.  First, suppose the attack occurs at $x \in X$.  If $z \in X$, then $z$ sends one to $x$.  If $z \in Y$, then $N(z)$ sends a total weight of $1$ to $x$, and $z$ redistributes its weight of $1$ to $N(z)$ so that $w$ is restored.  Suppose, then, that the attack occurs in $Y$.  If $z \in X$, then $N(y)$ sends its total weight of $1$ to $y$, and $z$ sends its weight of $1$ to $N(y)$ to restore $w$.  If $z \in Y$, then $N(y)$ sends a total weight of $1$ to $y$, $N(z)$ sends a total weight of $1$ to $N(y)$ to restore $w$ on those vertices (note that if $N(y) \cap N(z) \neq \emptyset$, then some vertex may send weight to itself), and $z$ sends its weight of $1$ to $N(z)$ to restore $w$ on those vertices.
\end{proof}

Recall the definition of $F(G)$ from the Section \ref{sec:behaviour}:
$F(G) = \max\{f(v): v \in V\}$ where, 
for each vertex $v \in V(G)$, the quantity $f(v)$ is the least total weight of a fractional dominating function in which the
weight assigned to vertex $v$ is at least 1.

\begin{prop}
Let $G$ be a split graph.  Then $\gamma_f^\infty(G) = F(G)$.
\label{dominate_split}
\end{prop}

\begin{proof}
Let $G$ be a split graph.
Let $X \cup Y$ be a partition of $V$ for which $X$ is a clique and $Y$ is an independent set.

We will show that $G$  can be eternally fractionally dominated by $n$ f.d.-functions.
For each  $v \in V$, let $D_v$ be a  fractional dominating set of $G$ of least total weight in which the
weight assigned to vertex $v$ is at least 1.
Since any excess weight assigned to $v$ can be arbitrarily redistributed to its neighbours in $X$, 
without loss of generality every vertex in $Y$, except $v$ if $v \in Y$,  has weight 0 in $D_v$.  

Consider the fractional dominating functions $D_u$ and $D_w$.
To reconfigure $D_u$ to $D_w$, simultaneously: 
(i) the vertex $u$ sends its weight to its neighbours; 
(ii)  the vertex $w$ receives weight 1 from its neighbours;
(iii) the weights assigned to vertices in $X$ and neither used in (ii) nor assigned in (i) are redistributed per $D_w$ (which is possible).

Thus, the set $D_u$ can be reconfigured to $D_w$  for all $u, w \in V$.
Hence $G$ can be fractionally eternally dominated by $n$ f.d.-functions, each of total weight $F(G)$.
Therefore $\gamma_f^\infty = F(G)$, and the proof is complete.
\end{proof}

\begin{cor}
If $G$ is a split graph, then $\gamma_f^\infty$
and a strategy for eternally fractionally dominating $G$ can be computed in polynomial time.
Further, $\gamma_f^\infty$ is a rational number.
\label{comp_efd_split}
\end{cor}

By contrast, the problem of deciding whether a given Hamiltonian split graph has eternal domination number at most a given integer $k$ is NP-complete \cite{BDEMY17}.

We say that a vertex $v$ in a graph $G$ is \textit{fractionally domination-critical} (or f.d.-critical for short) if $\gamma_f(G - v) < \gamma_f(G)$.  We define a vertex $v$ to be \textit{fully f.d.-critical} if $\gamma_f(G)- \gamma_f(G - v) = 1$.  Note that ``fully'' is not a vacuous addition; if $G=\mathcal{G}_{t,d}$ for $t,d$, then $G$ is an example of a graph containing at least one f.d.-critical vertex $v$ for which $\gamma_f(G) - \gamma_f(G - v) < 1$.


\begin{prop}
Let $G$ be a split graph which is not complete.
Let $X \cup Y$ be a partition of $V$ such that $X$ is a clique and $Y$ is an independent set.
Then $\gamma_f^\infty = \gamma_f$ if and only if 
every vertex in $X$ is adjacent to a vertex in $Y$ and
every vertex in $Y$ is fully f.d.-critical.
\end{prop}
\begin{proof}
Note that, since $G$ is not complete, the set $Y \neq \emptyset$.

Suppose $\gamma_f^\infty = \gamma_f \geq 2$.  
We will show that every vertex in $X$ is adjacent to a vertex in $Y$ and
every vertex in $Y$ is fully f.d.-critical.

First suppose there exists a vertex $x \in X$ such that $N(x) \cap Y = \emptyset$. 
Let $D_x$ be a fractional eternal dominating set of $G$ in which the weight assigned to $x$ is at least 1.
Without loss of generality, all vertices in $Y$ are assigned weight 0 in $D_x$.
But then the restriction of $D_x$ to $G-x$ is a fractional dominating set of $G-x$, and hence of $G$ (since $x \in X$), 
contrary to the hypothesis that $\gamma_f^\infty = \gamma_f$.
Therefore, every vertex in $X$ is adjacent to a vertex in $Y$. 

Let $y \in Y$.
Then, for any vertex $y \in Y$ there exists a fractional (eternal) dominating set $D_y$ in which the weight 
assigned to $y$ is 1 (any excess weight assigned to $y$ can be arbitrarily redistributed to its neighbours).
Since $N(y) \subseteq X$, the set $D_y - \{y\}$ is a fractional dominating set of $G - y$, it follows that $y$ is fully f.d.-critical.

We now prove the converse.
Now suppose every vertex in $X$ is adjacent to a vertex in $Y$ and every vertex in $Y$ is fully f.d.-critical.
Then, for any $y \in Y$ there exists a fractional dominating set in which the weight assigned to $y$ equals $1$.  To see this, take a fractional dominating set of $G - y$ and extend it to a fractional dominating set of $G$ by assigning weight $1$ to $y$.  This is a fractional dominating set of $G$ with total weight $\gamma_f$ since $y$ is fully f.d.-critical.
For any $x \in X$ there exists a fractional dominating set in which the weight assigned to $x$ is at least 1 -- let $y$ in $N(x) \cap Y$ and the fractional dominating set as above with weight 1 assigned to $y$, add the weight assigned to $y$ to the weight assigned to $x$, and assign $y$ weight 0.  
Thus $\gamma_f(G) = F(G)$, and the result follows from Proposition \ref{dominate_split}.
\end{proof}



\subsection{Packings and Strongly Chordal Graphs}

A \textit{distance-$2$ vertex packing} (henceforth referred to as a \textit{$2$-packing}) of a graph $G$ is a set $P \subseteq V(G)$ such that the distance from $x$ to $y$ is at least $3$ for every distinct $x,y \in P$.  Clearly, $\fed(G) \geq \max\{|P| \,:\, P \textup{ is a $2$-packing in $G$}\}$.  

\begin{prop} \label{prop:packing}
If $P$ is a $2$-packing in a graph $G$ such that $N[P] \subsetneq V(G)$, then $\fed(G) \geq 1 + |P|$.
\end{prop}

\begin{proof}
    Let $v \in V(G) \setminus N[P]$.  Suppose $v$ is attacked.  Then the resulting fractional dominating function
    must have weight at least one on $v$, and weight at least one in the neighbourhood of each vertex in the 2-packing.
\end{proof}

A dominating set $D$ is called \textit{efficient} if $|N[v] \cap D| = 1$ for every $v \in V(G)$.  A dominating set $D$ is \textit{near-efficient} if there exists a vertex $v \in V(G) \setminus D$ such that $D$ is an efficient dominating set in $G-v$.

\begin{cor}\label{packcor1}
If $G$ has an efficient dominating set, then $\fed(G) \geq \gamma(G)$.
\end{cor}

\begin{cor}\label{packcor2}
If the maximum possible value of $|P|$ is $\gamma(G)$, then $\fed(G) \geq \gamma(G)$.
\end{cor}

Since all strongly chordal graphs satisfy the conditions of Corollary \ref{packcor2} (see \cite{farber1984domination}), we obtain:

\begin{cor}
If $G$ is a strongly chordal graph, then $\fed(G) \geq \gamma(G)$.
\end{cor}

We will revisit the notion of $2$-packings in Sections \ref{ssec:Cprism} and \ref{ssec:Mprism}.

\subsection{Kneser Graphs}

Finally, we consider Kneser graphs.  The Kneser graph $KG_{n,k}$ is the graph whose vertex set consists of all $k$-subsets of an $n$-set and where vertices are adjacent if and only if they are disjoint.  We determine the exact value of $\fed(KG_{n,k})$ for the case when $k=2$; the value of $\fed(G)$ for the Petersen graph, which is $KG_{5,2}$,  follows as a special case.  We start with the following theorem.

\begin{thm}\label{thm:petersen}
For $\fed(KG_{5,2}) = 3$.
\end{thm}

\begin{proof}
	Suppose the $n$-set is $\{1, 2, 3, 4, 5\}$ and without loss of generality $w(\{1, 2\}) \geq 1$. Let $N_1$ and $N_2$ be respectively the set of all neighbours and non-neighbours of $\{1, 2\}$. Since the sum of the weights in the closed neighbourhood of each vertex of $N_2$ is greater than or equal to $1$, we obtain $\sum_{v \in N_2} \sum_{u \in N[v]} w(u) \geq 6$. In the preceding sum, the weight on each vertex of $N_2$ is repeated $3$ times and the weight on each vertex of $N_1$ is repeated $2$ times. So we have: $3 \times \sum_{v \in N_1 \cup N_2} w(v) + 3 \times w(\{1, 2\}) \geq 6 + 3 \implies 3 \times \sum_{v \in V} w(v) \geq 9 \implies \sum_{v \in V} w(v) \geq 3$. Now, it remains to prove $\fed(KG_{5,2}) \leq 3$. To this end, we first place a weight of $1$ on the vertex $\{1, 2\}$ and a weight of $\frac{1}{3}$ on each vertex of $N_2$. To respond to an attack on a vertex of $N_1$ (say $\{3, 4\}$ without loss of generality), we move the weight of $1$ from $\{1, 2\}$ to $\{3, 4\}$ and the weights from $N(\{3, 4\}) \backslash N(\{1, 2\})$ to $N(\{1, 2\}) \backslash N(\{3, 4\})$ along the following disjoint paths: $\{1, 5\}-\{2, 4\}-\{3, 5\}$ and $\{2, 5\}-\{1, 3\}-\{4, 5\}$. If a vertex of $N_2$ is attacked (say $\{1, 3\}$ without loss of generality), we move the weights from $\{2, 4\}$ and $\{2, 5\}$ to it and we share the weight of $1$ on the vertex $\{1, 2\}$ equally among the vertices $\{1,2\}, \{3,4\}$ and $\{3,5\}$.
\end{proof}

\begin{thm}\label{thm:kneser}
For every integer $n \geq 6$, $\fed(KG_{n,2}) = 1+\frac{n-2}{n-4}=2+\frac{2}{n-4}=\frac{2n-6}{n-4}$.
\end{thm}

\begin{proof}
    Suppose the $n$-set is $\{1, 2, ..., n\}$ and without loss of generality $w(\{1, 2\}) \geq 1$. Let $N_1$ and $N_2$ be respectively the set of all neighbours and non-neighbours of $\{1, 2\}$. Since the sum of the weight in the closed neighbourhood of each vertex of $N_2$ is greater than or equal to $1$, we obtain $\sum_{v \in N_2} \sum_{u \in N[v]} w(u) \geq 2(n-2)$.  Consider a vertex $\{a,b\}$ in $N_2$.  Note that $\{1,2\}$ and $\{a,b\}$ share precisely one element; without loss of generality suppose $a=1$.  There are $n-3$ sets which are disjoint from $\{1,b\}$ and intersect $\{1,2\}$.  Thus the weight of $\{1,b\}$ (and similarly for every vertex in $N_2$) is counted $n-2$ times in the preceding sum.  Now consider a vertex $\{a,b\} \in N_1$.  The number of sets which intersect $\{1,2\}$ but do not intersect $\{a,b\}$ is $2(n-4)$, and thus the weight on each vertex of $N_1$ is repeated $2(n-4)$ times. So we have: $2(n-4)(\sum_{v \in N_1} w(v) + \sum_{v \in N_2} w(v)) \geq 2(n-2) \implies \sum_{v \in N_1} w(v) + \sum_{v \in N_2} w(v) \geq \frac{n-2}{n-4} \implies \fed(KG_{n,2}) \geq 1+\frac{n-2}{n-4}$. 
    
    Now, it remains to prove $\fed(KG_{n,2}) \leq 1+\frac{n-2}{n-4}$. To this end, we first place a weight of $1$ on the vertex $\{1, 2\}$ and a weight of $\frac{1}{{n-3 \choose 2}}$ on each vertex of $N_1$.  This fractionally dominates the graph since, for every set $S$ which intersects $\{1,2\}$ (aside from $\{1,2\}$ itself), there are ${n-3 \choose 2}$ ways to construct a set which is disjoint from both that set and $\{1,2\}$.  We will show that, after any attack, we may reconfigure the fractional dominating set to have weight $1$ on the attacked vertex, weight $\frac{1}{{n-3 \choose 2}}$ on each neighbour of the attacked vertex, and weight $0$ on all other vertices.
    
    To respond to an attack on a vertex of $N_1$ (say $\{3, 4\}$ without loss of generality), we first move a weight of $\frac{{n-3 \choose 2}-1}{{n-3 \choose 2}}$ from $\{1, 2\}$ to $\{3, 4\}$.  The subgraph induced by the edges with one end in $N(\{1, 2\}) \backslash N(\{3, 4\})$ and the other in $N(\{3, 4\}) \backslash N(\{1, 2\})$ is a regular bipartite graph.  By Hall's Theorem there is a perfect matching between $N(\{1, 2\}) \backslash N(\{3, 4\})$ and $N(\{3, 4\}) \backslash N(\{1, 2\})$, and so the weights from $N(\{1, 2\}) \backslash N(\{3, 4\})$ can be moved to $N(\{3, 4\}) \backslash N(\{1, 2\})$.  Suppose now that a vertex of $N_2$ is attacked (say $\{1, 3\}$ without loss of generality).  Move the total weight of $1$ in the neighbourhood of that vertex to it, and share the weight of $1$ on the vertex $\{1, 2\}$ equally among its ${n-3 \choose 2}$ common neighbours with $\{1, 3\}$.  Now, the vertices of $N(\{1, 2\}) \backslash N(\{1, 3\})$ are the sets of the form $\{3,x\}$ where $x \in \{4,\ldots,n\}$.  Similarly, the vertices of $N(\{1, 3\}) \backslash N(\{1, 2\})$ are the sets of the form $\{2,x\}$ where $x \in \{4,\ldots,n\}$.  The edges with one end in each set induce a regular bipartite graph so, as above, the vertices in $N(\{1, 2\}) \backslash N(\{1, 3\})$ can pass their weights to the vertices in $N(\{1, 3\}) \backslash N(\{1, 2\})$ along a perfect matching.  In either case, we finish with a fractional dominating set with weight $1$ on the attacked vertex, weight $\frac{1}{{n-3 \choose 2}}$ on each neighbour of the attacked vertex, and weight $0$ on all other vertices.
\end{proof}

\section{Connectivity and Cayley graphs}\label{sec:Cayley}

The main results in the section are focused on $\fed(G)$ for Cayley graphs.  To begin, we examine the relationship between the vertex connectivity of a graph $G$ (denoted $\kappa(G)$), the degrees of its vertices, and $\fed(G)$.  Recall that $\delta(G)$ and $\Delta(G)$ denote the minimum and maximum degree of $G$, respectively.

The following lemma follows from results in \cite{domke1988fractional} and \cite{grinstead1990fractional}.

\begin{lem}\label{degreebounds}
If $G$ is an $n$-vertex graph with $\delta(G) = \delta$ and $\Delta(G) = \Delta$, then $\frac{n}{\Delta+1} \leq \gamma_f(G) \leq \frac{n}{\delta+1}$.
\end{lem}

This lemma also inspires the following result on $\fed(G)$, which gives a general bound on $\fed(G)$ depending on the order of $G$ and its connectivity.

\begin{thm}\label{conn}
If $G$ has connectivity $\kappa$, then $\fed(G) \leq \frac{n + \kappa}{\kappa+1}$.
\end{thm}

\begin{proof}
Let $w(x)$ denote the weight of a vertex $x$.  Begin by weighting an arbitrary vertex $y$ with $1$ and all other vertices $\frac{1}{\kappa+1}$.  Clearly $\sum_{u \in N[v]}w(u) \geq (\delta+1)\frac{1}{\kappa+1} \geq 1$ for every vertex $v \in V(G)$.  We give a simple strategy to show that we can maintain a weighting where the attacked vertex receives weight $1$ while all others have $\frac{1}{\kappa+1}$.  If a vertex other than $y$ is attacked, say $z$, choose $\kappa$ internally disjoint $yz$-paths $P_1, \ldots, P_\kappa$.  For each $P_i = yv_1v_2\cdots v_tz$, each $v_i$ will send weight $\frac{1}{\kappa+1}$ to $v_{i+1}$ (considering $y$ as $v_0$ and $z$ as $v_{t+1}$).  In this way, $z$ receives weight $1$ and every other vertex has weight $\frac{1}{\kappa+1}$.
\end{proof}

Following immediately from Proposition \ref{prop:basicbound}, Lemma \ref{degreebounds}, and Theorem \ref{conn}, we obtain the following corollary, which motivates our study of Cayley graphs.

\begin{cor}\label{regconn}
If $G$ is a $d$-regular, $d$-connected graph, then $\frac{n}{d+1} \leq \fed(G) \leq \frac{n + d}{d+1}$.
\end{cor}

\begin{cor}
If $G$ is $d$-regular and $d$-connected, then there exists a polynomial time approximation algorithm for
$\gamma_f^\infty$ that has error at most 1.
\end{cor}
\begin{proof}
The proof of Theorem \ref{conn} 
shows that $G$ can be fractionally eternally dominated by $n$ f.d.-functions, each of
total weight less that $\gamma_f + 1$.  The result then follows from Proposition \ref{nsets}.
\end{proof}

Every abelian Cayley graph is regular and has connectivity equal to its degree, and thus we obtain the following:

\begin{cor}
If $G$ is an abelian Cayley graph, then $\fed(G) - \gamma_f(G) < 1$.
\end{cor}

In light of this, it is reasonable to ask whether or not $\fed(G)$ can be exactly determined for abelian Cayley graphs.  The rest of this section is devoted to showing that this is a difficult task, even under the strong assumption that the graph is cubic.

\subsection{Cubic abelian Cayley graphs}

At the end of Section \ref{sec:basic}, we gave a number of conditions under which $\gamma(G)$ is a lower bound on $\fed(G)$.  In the case of cubic abelian Cayley graphs, however, it turns out that $\gamma(G) = \med(G)$, and thus $\gamma(G)$ is an upper bound on $\fed(G)$.  We characterize precisely those cubic abelian Cayley graphs for which the upper bound is strict.

\begin{thm} \label{thm:Cayleyfrac}
If $G$ is a cubic abelian Cayley graph, then $\fed(G) \leq \gamma(G) = \med(G)$.  Furthermore, $\fed(G) < \gamma(G)$ if and only if $G$ is isomorphic either to $C_{4k+2} \Box K_2$ or to $\Cay(\Z_{8k},\{\pm 1,4k\})$ for some integer $k \geq 1$.
\end{thm}

To prove this theorem we rely on the following characterization of cubic abelian Cayley graphs found in \cite{wai19953}.  Recall that $Q_d$ denotes the hypercube of dimension $d$.

\begin{thm}\label{thm:Cayleyclassification}
A graph $G$ is a cubic abelian Cayley graph if and only if it is one of the following:
	\begin{compactenum}
	\item $K_4$,
    \item $Q_3$,
    \item $C_n \Box K_2$ ($n \geq 3$),
    \item $\Cay(\Z_{2n}, \{\pm 1,n\})$ ($n \geq 3$).
	\end{compactenum}
\end{thm}

We settle the first two cases quickly.

\begin{lem}\label{lem:Cayleysimple}
If $G \in \{K_4,Q_3\}$, then $\fed(G) = \gamma(G)$
\end{lem}

\begin{proof}
If $G = K_4$, then each parameter is clearly equal to $1$.  Suppose $G = Q_3$, which has $\gamma(G) = 2$.  On one hand, $G$ has an efficient dominating set of size $2$, and so by Corollary \ref{packcor1} we have that $\fed(G) \geq \gamma(G) = 2$. On the other hand, $\fed(G) \leq \gamma_{\rm m}^\infty(G) = 2 = \gamma(G)$, therefore equality holds.
\end{proof}

To finish the proof of Theorem \ref{thm:Cayleyfrac}, the final two cases require a more in-depth analysis.

\subsection{Cyclic prisms}\label{ssec:Cprism}

Let us now turn to $C_n \Box K_2$, where $n \geq 3$.  We first show that the all-guards move model of eternal domination requires a number of guards equal to the domination number, and then consider $\fed(G)$ by cases based on the value of $n \pmod 4$.

\begin{lem}\label{lem:allguardsprism}
For each integer $n \geq 3$, $\med(C_n \Box K_2) = \gamma(C_n \Box K_2)$.
\end{lem}

\begin{proof}
Suppose the guards are on the vertices of a dominating set of minimum size which contains the vertex $(0,0)$ and without loss of generality the guard on $(0,0)$ is the one who responds to the first attack. If he moves to $(0,1)$, then all guards on $(i,j)$ can move to $(i,j+1)$ in order to maintain a similar configuration (after relabelling the vertices of the graph, and where operations in the first and second coordinates are taken mod $n$ and mod $2$, respectively). If he moves to $(1,0)$, then all guards on $(i,j)$ can move to $(i+1,j)$ to maintain a similar configuration. If he moves to $(n-1,0)$, then all guards on $(i,j)$ can move to $(i-1,j)$.
\end{proof}

\begin{lem}\label{lem:prismfrac1}
If $n \equiv 0,1,3 \pmod 4$, then $\fed(C_n \Box K_2) = \gamma(C_n \Box K_2) = \lceil \frac{n}{2} \rceil$.
\end{lem}

\begin{proof}
By Lemma \ref{lem:allguardsprism}, it suffices to show that $\gamma(C_n \Box K_2) = \lceil \frac{n}{2} \rceil$ is a lower bound on $\fed(C_n \Box K_2)$.  We proceed by cases:
\begin{compactenum}
\item $n \equiv 0 \pmod 4$: The set $S_0=\{(4i,0): i \in [\frac{n}{4}]\} \cup \{(4i+2,1): i \in [\frac{n}{4}]\}$ (where $[\frac{n}{4}]=\{1, \ldots, \frac{n}{4}\}$) is an efficient dominating set of cardinality $\frac{n}{2}$. Then, by Corollary \ref{packcor1} we have $\fed(C_n \Box K_2) \geq \lceil \frac{n}{2} \rceil$.

\item $n \equiv 1 \pmod 4$: The set $S_1=\{(4i,0): i \in [\frac{n-1}{4}]\} \cup \{(4i+2,1): i \in [\frac{n-1}{4}]\}$ is a $2$-packing of cardinality $\frac{n-1}{2}$ which does not dominate the vertex $(n-1,1)$. Hence, by Proposition \ref{prop:packing} we have $\fed(C_n \Box K_2) \geq 1+\frac{n-1}{2}=\lceil \frac{n}{2} \rceil$.

\item $n \equiv 3 \pmod 4$: The set $S_3=\{(4i,0): i \in [\frac{n+1}{4}]\} \cup \{(4i+2,1): i \in [\frac{n-3}{4}]\}$ is a $2$-packing of cardinality $\frac{n-1}{2}$ which does not dominate the vertex $(n-1,1)$. As a result, by Proposition \ref{prop:packing} we have $\fed(C_n \Box K_2) \geq 1+\frac{n-1}{2}=\lceil \frac{n}{2} \rceil$. \qedhere
\end{compactenum}
\end{proof}

\begin{lem}\label{lem:prismfrac2}
If $n \equiv 2 \pmod 4$, then $\fed(C_n \Box K_2) < \gamma(C_n \Box K_2)$.
\end{lem}

\begin{proof}
Observe that $\fed(C_n \Box K_2) \leq \frac{2n+3}{4}$ follows from Theorem \ref{conn}. So, it suffices to prove that $\gamma(C_n \Box K_2) > \frac{2n+3}{4}$. To this end, we prove that $C_n \Box K_2$ does not contain any efficient dominating set. Suppose this is not true and let $S_2$ be an efficient dominating set of $C_n \Box K_2$. Without loss of generality, let $(0,0)$ be a vertex of $S_2$. Since $S_2$ is a $2$-packing, $S_2$ must contain the vertices $(2,1), (4,0)$ and all the vertices $\{(4i,0): i \in [\frac{n+2}{4}]\} \cup \{(4i+2,1): i \in [\frac{n-2}{4}]\}$. In this case the vertices $(0,0)$ and $(n-2,0)$ would be two vertices of $S_2$ with non disjoint neighbourhood (contradiction).
\end{proof}

\subsection{M\"obius prisms}\label{ssec:Mprism}

Lastly, we consider the so-called M\"obius prisms $\Cay(\Z_{2n}, \{\pm 1,n\})$ for $n \geq 3$.

\begin{lem}\label{lem:allguardsladder}
For each integer $n \geq 3$, $\med(\Cay(\Z_{2n}, \{\pm 1,n\})) = \gamma(\Cay(\Z_{2n}, \{\pm 1,n\}))$.
\end{lem}

\begin{proof}
Suppose the guards are on the vertices of a dominating set of minimum size which contains the vertex $0$ and, without loss of generality, the guard on $0$ is the one who responds to the first attack. If he moves to vertex $1$, then any guard on a vertex $i$ can move to the vertex $i+1$ in order to maintain a similar configuration (after relabelling the vertices of the graph). If he moves to vertex $2n-1$, then any guard on a vertex $i$ can move to the vertex $i-1$ to maintain a similar configuration. If he moves to vertex $n$, then any guard on a vertex $i$ can move to vertex $i+n$.
\end{proof}

\begin{lem}\label{lem:ladderfrac1}
If $n \not\equiv 0 \pmod 4$, then $\fed(\Cay(\Z_{2n}, \{\pm 1,n\})) = \gamma(\Cay(\Z_{2n}, \{\pm 1,n\})) = \lceil \frac{n}{2}  \rceil$.
\end{lem}

\begin{proof}
By Lemma \ref{lem:allguardsladder}, it suffices to show that $\gamma(\Cay(\Z_{2n}, \{\pm 1,n\})) = \lceil\frac{n}{2}  \rceil$ is a lower bound on $\fed(\Cay(\Z_{2n}, \{\pm 1,n\}))$.  We proceed by cases:

\begin{compactenum}
\item $n \equiv 1 \pmod 4$: The set $S_1=\{4i: i \in [\frac{n-1}{4}]\} \cup \{4i+n+2: i \in [\frac{n-1}{4}]\}$ is a $2$-packing of cardinality $\frac{n-1}{2}$ which does not dominate the vertex $n-1$. Then, by Proposition \ref{prop:packing} we have $\fed(\Cay(\Z_{2n}, \{\pm 1,n\})) \geq 1+\frac{n-1}{2}=\lceil \frac{n}{2} \rceil$.

\item $n \equiv 2 \pmod 4$: Now, the set $S_2=\{4i: i \in [\frac{n+2}{4}]\} \cup \{4i+n+2: i \in [\frac{n-2}{4}]\}$ is an efficient dominating set of cardinality $\frac{n}{2}$. Hence, by Corollary \ref{packcor1} we have $\fed(\Cay(\Z_{2n}, \{\pm 1,n\})) \geq \frac{n}{2}$.

\item $n \equiv 3 \pmod 4$: Finally, the set $S_3=\{4i: i \in [\frac{n+1}{4}]\} \cup \{4i+n+2: i \in [\frac{n-3}{4}]\}$ is a $2$-packing of cardinality $\frac{n-1}{2}$ which does not dominate the vertex $n-1$. As a result, by Proposition \ref{prop:packing} we have $\fed(C_n \Box K_2) \geq 1+\frac{n-1}{2}=\lceil \frac{n}{2} \rceil$. \qedhere
\end{compactenum}
\end{proof}

\begin{lem}\label{lem:ladderfrac2}
If $n \equiv 0 \pmod 4$, then $\fed(\Cay(\Z_{2n}, \{\pm 1,n\})) < \gamma(\Cay(\Z_{2n}, \{\pm 1,n\}))$.
\end{lem}

\begin{proof}
Observe that $\fed(\Cay(\Z_{2n}, \{\pm 1,n\})) \leq \frac{2n+3}{4}$ follows from Theorem \ref{conn}. So, it suffices to prove that $\gamma(C_n \Box K_2) > \frac{2n+3}{4}$. To this end, we prove that $\Cay(\Z_{2n}, \{\pm 1,n\})$ does not contain any efficient dominating set. Suppose this is not true and let $S_0$ be an efficient dominating set of $\Cay(\Z_{2n}, \{\pm 1,n\})$. Without loss of generality, let $0$ be a vertex of $S_0$. Since $S_0$ is a $2$-packing, $S_0$ must contain the vertices $n+2$ and all the vertices $\{4i: i \in [\frac{n}{4}]\} \cup \{4i+n+2: i \in [\frac{n}{4}]\}$, in which case the vertices $0$ and $n-2$ would be two vertices of $S_0$ with non disjoint neighbourhood.
\end{proof}

We may now complete the proof of Theorem \ref{thm:Cayleyfrac}.

\begin{proof}[Proof of Theorem \ref{thm:Cayleyfrac}]
The result now follows immediately as a consequence of Theorem \ref{thm:Cayleyclassification} and Lemmas \ref{lem:Cayleysimple}, \ref{lem:prismfrac1}, \ref{lem:prismfrac2}, \ref{lem:ladderfrac1}, and \ref{lem:ladderfrac2}.
\end{proof}

\subsection{Further results on cyclic and M\"obius prisms}

Determining the exact values for the exceptional cases of Theorem \ref{thm:Cayleyfrac} appears to be surprisingly difficult, and we leave this as an open problem for future research. We conclude with a look at three particular cases of cyclic and M\"obius prisms -- $C_{6} \Box K_2$, $C_{10} \Box K_2$, and $Cay(\mathbb{Z}_{8}, \{\pm 1, 4\})$ -- whose exact values of $\fed$ are not given by Theorem \ref{thm:Cayleyfrac}, and obtain lower bounds for general graphs from some of these exceptional cases.

\begin{figure}[h!]
    \centering
\scalebox{0.7}{
\begin{tikzpicture}

\tikzset{vertex/.style = {shape=circle,draw,minimum size=1.2em}}
\tikzset{edge/.style = {> = latex'}}

\node[vertex,label=above:{$1$}] (a) at (7,5) {$a$};
\node[vertex,label=left:{$0$}] (b) at  (2,3) {$b$};
\node[vertex,label=left:{$x$}] (c) at  (2,-3) {$c$};
\node[vertex,label=below:{$1/2 - x$}] (d) at (7,-5) {$d$};
\node[vertex,label=right:{$x$}] (e) at (12,-3) {$e$};
\node[vertex,label=right:{$0$}] (f) at (12,3) {$f$};

\node[vertex,label=below:{$1/2$}] (g) at (7,3.5) {$g$};
\node[vertex,label=right:{$0$}] (h) at  (4,2.25) {$h$};
\node[vertex,label=right:{$1/2$}] (i) at (4,-2.25) {$i$};
\node[vertex,label=above:{$1/2 - x$}] (j) at (7,-3.5) {$j$};
\node[vertex,label=left:{$1/2$}] (k) at (10,-2.25) {$k$};
\node[vertex,label=left:{$0$}] (l) at  (10,2.25) {$l$}; 


\draw[edge] (a) to (b);
\draw[edge] (b) to (c);
\draw[edge] (c) to (d);
\draw[edge] (d) to (e);
\draw[edge] (e) to (f);
\draw[edge] (f) to (a);

\draw[edge] (g) to (h);
\draw[edge] (h) to (i);
\draw[edge] (i) to (j);
\draw[edge] (j) to (k);
\draw[edge] (k) to (l);
\draw[edge] (l) to (g);

\draw[edge] (a) to (g);
\draw[edge] (f) to (l);
\draw[edge] (e) to (k);
\draw[edge] (d) to (j);
\draw[edge] (c) to (i);
\draw[edge] (b) to (h);

\end{tikzpicture}
}
    \caption{An initial weighting of $C_6 \Box K_2$}
    \label{fig:C6K2}
\end{figure}

It can be checked using an LP solver (or tedious manual calculations), that a solution to the initial configuration LP requires a total weight of at least $7/2$.  We now prove that this is insufficient.


\begin{thm}
$\frac{7}{2} < \fed(C_6 \Box K_2) \leq 4$
\end{thm}

\begin{proof}
Let $G = C_6 \Box K_2$, with the vertices labelled as in Figure \ref{fig:C6K2}.  The upper bound follows from Theorem \ref{conn}, as $G$ is $3$-regular and $3$-connected.  the fact that $\med(G) = 4$.  Suppose that $\fed(C_6 \Box K_2) = \frac{7}{2}$ and let $w$ be a feasible initial weight function.  For a set $S \subseteq V(G)$, we let $w(S) = \sum_{v \in S} w(v)$.  Suppose, without loss of generality, that $a$ is the vertex to initially receive weight $1$.  The next five claims show that the only possible initial weightings are those given in Figure \ref{fig:C6K2}.\\

\noindent \textit{Claim 1:} $w(g) \geq \frac{1}{2}$. \\
\begin{proofclaim}
The total weight assigned to $V(G) \setminus \{a\}$ is $\frac{5}{2}$.  Furthermore, $w(N[d]) + w(N[h]) + w(N[l]) \geq 3$, and so $w(g) \geq \frac{1}{2}$.
\end{proofclaim}

\noindent \textit{Claim 2:} $w(b) = w(f) = w(h) = w(l) = 0$. \\
\begin{proofclaim}
 First note that $w(N[c]) + w(N[e]) + w(N[i]) + w(N[k]) \geq 4$, and that the sum on the left counts the weight of every vertex in $\{c,d,e,i,j,k\}$ twice.  It follows that
\begin{align*}
   & w(N[c]) + w(N[e]) + w(N[i]) + w(N[k]) + 2w(g) \geq 5 \\
    \implies& 2w(V(G) \setminus \{a\}) - [w(b) + w(f) + w(h) + w(l)]\geq 5 \\
    \implies& 5 - [w(b) + w(f) + w(h) + w(l)]\geq 5 \\
    \implies& w(b) = w(f) = w(h) = w(l) = 0
\end{align*}
as desired.
\end{proofclaim}

\noindent \textit{Claim 3:} $w(g) = \frac{1}{2}$. \\
\begin{proofclaim}
Now, we have that
$$ w(c) + w(d) + w(e) + w(g) + w(i) + w(j) + w(k) = \frac{5}{2}$$
or, equivalently,
$$w(N[i]) + w(N[e]) + w(g) = \frac{5}{2}.$$
However, since $w(N[i])$ and $w(N[e])$ must each be at least $1$, and $N[i]$ and $N[e]$ are disjoint, we have that $w(g) \leq \frac{1}{2}$ and so $w(g) = \frac{1}{2}$.
\end{proofclaim}

\noindent \textit{Claim 4:} $w(i) = w(k) = \frac{1}{2}$. \\
\begin{proofclaim}
By considering $N(h)$ and $N(l)$, we see that $w(i) \geq \frac{1}{2}$ and $w(k) \geq \frac{1}{2}$.  However, since $w(N[d]) \geq 1$ and $w(N[d]) + w(i) + w(k) = 2$, we have that $w(i) = w(k) = \frac{1}{2}$.
\end{proofclaim}

\noindent \textit{Claim 5:} For some $x \in [0,\frac{1}{2}]$, $w(c) = w(e) = x$ and $w(d) = w(j) = \frac{1}{2} - x$. \\
\begin{proofclaim}
By considering $N(c), N(e), N(i),$ and $N(k)$, we get that each of the $w(c) + w(d)$, $w(d) + w(e)$, $w(c) + w(j)$, $w(e) + w(j)$ is at least $1/2$.  
\end{proofclaim}

Now, we show that, for any value of $x \in [0,\frac{1}{2}]$, there is a vertex in $G$ which can only be responded to in such a way as to no longer have a fractional dominating function on $V(G)$.


Suppose $x \in (0, \frac{1}{2}]$ and consider an attack on vertex $c$. We need to move the weight of $\frac{1}{2}$ from vertex $i$ and the weight of $\frac{1}{2}-x$ from vertex $d$ to vertex $c$ in order to respond to the attack. Since there is only a weight of $\frac{1}{2}-x$ on $j$, it is impossible to maintain a weight of $\frac{1}{2}$ on vertex $i$ after the response to the attack (which contradicts Claim $1$).

Suppose now that $x=0$ and consider an attack on vertex $j$. According to Claims $1-5$, there must be a weight of $1$ on $j$, $\frac{1}{2}$ on $d$ and $0$ on each of the vertices $c, e, h, i, k, l$ after a response to the attack. However, since vertex $j$ has a total weight of $2$ in its closed neighbourhood, the total weight on the set of vertices at distance at most $2$ from vertex $j$ will be at least $2$. Therefore, it is impossible to maintain a weight of $1$ on $j$, $\frac{1}{2}$ on $d$ and $0$ on each of the vertices $c, e, h, i, k, l$ after any response to the attack.
\end{proof}

\begin{thm}
If $n \equiv 10 \pmod {12}$, then $\fed(C_n \Box K_2) \geq \frac{(n+2)(n+4)}{2(n+5)}$.
\end{thm}

\begin{proof}
It is known from Corollary \ref{regconn} that $\frac{2n+3}{4}$ is an upper bound on the fractional eternal domination number of $C_n \Box K_2$. Let $\epsilon \geq 0$ be a real number such that a total weight of $\frac{2n+3}{4}-\epsilon$ can dominate $C_n \Box K_2$. Let us consider an initial feasible weight function $w$. We may assume without loss of generality that vertex $v_0$ receives a weight of $1$. Let $S=\{v_{4i-1}: i \in [\frac{n}{4}]\} \cup \{u_{4i-3}: i \in [\frac{n}{4}]\}$, then $\sum_{v \in S} w(N[v])=(\sum_{v \in V} w(v))-w(v_0)-w(u_0)$. Since $\sum_{v \in S} w(N[v]) \geq \frac{n}{2}$, we have $(\sum_{v \in V} w(v))-w(v_0)-w(u_0) \geq \frac{n}{2} \implies w(u_0) \geq \frac{1}{4}+\epsilon$. This means that for any integer $i \in [\frac{n-1}{3}]$, the sum of the weight in the neighbourhood of the set $\{v_{3i-1}, u_{3i-1}\}$ is at least $\frac{5}{4}+\epsilon$. Hence, $(\frac{n-1}{3})(1+\frac{1}{4}+\epsilon) \leq \frac{2n+3}{4}-\epsilon-1-\frac{1}{4}-\epsilon \implies \epsilon \leq \frac{n-1}{4(n+5)}$. As a result, $\fed(C_n \Box K_2) \geq (\frac{n+2}{3})(1+\frac{1}{4}+\frac{(n-1)}{4(n+5)})=\frac{(n+2)(n+4)}{2(n+5)}$. 
\end{proof}

For the specific $n=10$ case, an exact value can be computed.

\begin{thm}
$\fed(C_{10} \Box K_2) = \frac{28}{5}$.
\end{thm}

\begin{proof}
Let $\epsilon \geq 0$ be such that a total weight of $\frac{28}{5}-\epsilon$ can dominate the graph. We may assume without loss of generality that vertex $v_0$ receives a weight of $1$. Now, $w(N[u_1])+w(N[v_3])+w(N[u_5])+w(N[v_7])+w(N[u_9]) \geq 5 \implies w(V \backslash \{v_0\})+w(u_0) \geq 5 \implies w(u_0) \geq 5-\frac{23}{5}+\epsilon=\frac{2}{5}+\epsilon$. If there is an attack on vertex $v_2$, we must move a weight of $1$ to that vertex and a weight of $\frac{2}{5}+\epsilon$ to vertex $u_2$. So, there must be a total weight of at least $\frac{7}{5}+\epsilon$ in the neighbourhood of the vertices $v_2$ and $u_2$. Since the same argument holds for the vertices $v_5$ and $v_8$, we have $3(\frac{7}{5}+\epsilon) \leq \sum\limits_{i=1}^9 w(v_i) + w(u_i)
\leq \frac{28}{5}-\epsilon-1-\frac{2}{5}-\epsilon=\frac{21}{5}-2\epsilon \implies 5\epsilon \leq 0$. On the other hand, Figures \ref{Figure:C10}, \ref{Figure:C10-1} and \ref{Figure:C10-2} along with Table \ref{Table:C10} (see the Appendix) show an initial feasible weight function with total weight $\frac{28}{5}$ and a response to all possible attacks on the vertices of the graph.
\end{proof}

We move on to exact values and bounds for special classes of M\"obius prisms.

\begin{thm}
$\fed(Cay(\mathbb{Z}_{8}, \{\pm 1, 4\})) = \frac{8}{3}$.
\end{thm}

\begin{proof}
We first prove that $\fed(Cay(\mathbb{Z}_{8}, \{\pm 1, 4\}) \geq \frac{8}{3}$. Let $\epsilon \geq 0$ be such that total weight of $\frac{8}{3}-\epsilon$ can dominate the graph. We may assume without loss of generality that vertex $v_0$ receives a weight of $1$. Since the sum of the weights in the neighbourhood of the vertices $v_3$ and $v_5$ must sum to at least $1$, we have: $w(v_2)+w(v_3)+w(v_4)+w(v_7) \geq 1$ and $w(v_4)+w(v_5)+w(v_6)+w(v_1) \geq 1$. Since there is a weight of $1$ on the vertex $v_0$, then 
$\sum\limits_{i=1}^7 w(v_i) \leq \frac{5}{3}-\epsilon$. 
So, $w(v_4)+ \sum\limits_{i=1}^7 w(v_i) \geq 2 \implies w(v_4) \geq \frac{1}{3}+\epsilon$. If there is an attack on vertex $v_2$, we must move a weight of $1$ to that vertex and a weight of at least $\frac{1}{3}+\epsilon$ to vertex $v_6$. So, the weight in the neighbourhood of vertex $v_2$ and $v_6$ must sum to at least $1+\frac{1}{3}+\epsilon$. Hence, $\frac{4}{3}+\epsilon \leq w(v_1)+w(v_2)+w(v_3)+w(v_5)+w(v_6)+w(v_7) \leq \frac{4}{3}-2\epsilon \implies 3\epsilon \leq 0$. Now, to prove that $\fed(Cay(\mathbb{Z}_{8}, \{\pm 1, 4\}) \leq \frac{8}{3}$, we place a weight of $1$ on vertex $v_0$, a weight of $\frac{1}{3}$ on vertex $v_4$ and a weight of $\frac{2}{3}$ on the vertices $v_2$ and $v_6$. The reader can check from Figure \ref{Figure:CayleyZ8} and Figure \ref{Figure:CayleyZ8-2} (see Appendix) that any attack on a vertex $v_i$ can be defended in a way such that vertex $v_i$ receives a weight of $1$, vertex $v_{i+4}$ receives a weight of $\frac{1}{3}$ and each of the vertices $v_{i+2}$, $v_{i+6}$ receive a weight of $\frac{2}{3}$.
\end{proof}




\begin{thm}
If $n \equiv 4 \pmod {12}$, then $\fed(Cay(\mathbb{Z}_{2n}, \{\pm 1, n\})) \geq \frac{(n+2)(n+4)}{2(n+5)}$.
\end{thm}

\begin{proof}
It is known from Corollary \ref{regconn} that $\frac{2n+3}{4}$ is an upper bound on the fractional eternal domination number of $Cay(\mathbb{Z}_{2n}, \{\pm 1, n\})$. Let $\epsilon \geq 0$ be a real number such that a total weight of $\frac{2n+3}{4}-\epsilon$ can fractionally eternally dominate $Cay(\mathbb{Z}_{2n}, \{\pm 1, n\})$. Let us consider an initial feasible weight function $w$. We may assume without loss of generality that vertex $v_0$ receives a weight of $1$. Let $S=\{v_{4i-1}: i \in [\frac{n}{4}]\} \cup \{v_{n+4i-3}: i \in [\frac{n}{4}]\}$. It follows from the preceding definition and from the fact that $n \equiv 0 \pmod 4$ that $N[v_i] \cap N[v_j] = \emptyset$ for any $i, j \in S, i \neq j$ unless $i=n-1$ and $j=n+1$. Now, $\sum_{v \in S} w(N[v])=(\sum_{v \in V} w(v))-w(v_0)-w(v_n)$. Since $\sum_{v \in S} w(N[v]) \geq \frac{n}{2}$, we have $(\sum_{v \in V} w(v))-w(v_0)-w(v_n) \geq \frac{n}{2} \implies w(v_n) \geq \frac{1}{4}+\epsilon$. This means that for any integer $i \in [\frac{n-1}{3}]$, the sum of the weight in the closed neighbourhood of the set $\{v_{3i-1}, v_{n+3i-1}\}$ is at least $\frac{5}{4}+\epsilon$. Since $n \equiv 1 \pmod 3$, $V-\{v_0, v_n\}$ can be partitioned into $\frac{n-1}{3}$ sets each of which contains the closed neighbourhood of $\{v_{3i-1}, v_{n+3i-1}\}$ for some $i \in [\frac{n-1}{3}]$. Hence, $(\frac{n-1}{3})(1+\frac{1}{4}+\epsilon) \leq \frac{2n+3}{4}-\epsilon-1-\frac{1}{4}-\epsilon \implies \epsilon \leq \frac{n-1}{4(n+5)}$. As a result, $\fed(Cay(\mathbb{Z}_{2n}, \{\pm 1, n\})) \geq (\frac{n+2}{3})(1+\frac{1}{4}+\frac{(n-1)}{4(n+5)})=\frac{(n+2)(n+4)}{2(n+5)}$. 
\end{proof}


\section{Graph products}\label{sec:products}

\subsection{Hypercubes}

In this final section, we consider $\fed(G)$ when $G$ is obtained by taking the Cartesian or strong product of two graphs. Perhaps the most relevant in the field of graph domination, due to its applications in coding theory, is the hypercube $Q_d$. Hypercubes are generally resistant to exact computation of domination parameters; the exact value of $\gamma(Q_d)$ has been determined for $d \leq 9$ and for $d=2^r-1$ for some positive integer $r$ but is generally open. It is known that $\gamma(Q_d) = \med(Q_d)$ since any attack on $Q_d$ can be defended by a guard shift, and so determining the number of guards needed in the all-guards move model is also generally open.  However, as a consequence of Corollary \ref{regconn}, we see that $\fed(Q_d)$ can at least be closely approximated.

\begin{thm}\label{thm:cube}
For any positive integer $d$, $\frac{2^d}{d+1} \leq \fed(Q_d) \leq \frac{2^d + d}{d+1}$.
\end{thm}

Figure \ref{fig:cube} compares the bounds from Theorem \ref{thm:cube} with the known values of $\gamma(Q_d)$ for small values of $d$ (see \cite{AHK17}). The equality in parameters for $d=1,3,7$ is not a coincidence. If $d = 2^r-1$ for some positive integer $r$, then $\gamma(Q_d) = \med(Q_d) = \frac{2^d}{d+1}$ because $Q_d$ has an efficient dominating set \cite{hill1986first}. Since this is a lower bound on $\fed(Q_d)$ (Theorem \ref{thm:cube}) as well as an upper bound (Proposition \ref{prop:basicbound}), equality holds.

\begin{figure}
    \begin{center}
\begin{tabular}{|c|c|c|}
\hline
$d$ & $\gamma_{\rm m}^{\infty}(Q_d) = \gamma(Q_d)$ & $\fed(Q_d)$ \\
\hline
1 & 1 & $1$ \\
2 & 2 & $[\frac{4}{3},2]$ \\
3 & 2 & $2$ \\
4 & 4 & $[\frac{16}{5},4]$  \\
5 & 7 & $[\frac{16}{3},\frac{37}{6}]$ \\
6 & 12 & $[\frac{64}{7},10]$ \\
7 & 16 & $16$ \\
8 & 32 & $[\frac{256}{9},\frac{88}{3}]$ \\
9 & 62 & $[\frac{256}{5},\frac{521}{10}]$  \\
10 & $[107,120]$ & $[\frac{1024}{11},94]$ \\
\hline
\end{tabular}
\end{center}
    \caption{Comparison of $\fed(Q_d)$ and $\gamma(Q_d)$}
    \label{fig:cube}
\end{figure}

\subsection{Grids}

\begin{thm}
For any integer $n \geq 1$, $\fed(P_n \Box P_2) = \lceil\frac{2n}{3}\rceil$.
\end{thm}

\begin{proof}
Since $C_{2n}$ is a spanning subgraph of $P_n \Box P_2$ and satisfies $\fed(C_{2n})=\lceil \frac{2n}{3} \rceil$, we have $\fed(P_n \Box P_2) \leq \lceil\frac{2n}{3}\rceil$. It remains to prove that $\fed(P_n \Box P_2) \geq \lceil\frac{2n}{3}\rceil$. To this end, we label the vertices of the graph $v_1, v_2, \ldots, v_n, v_1', v_2', \ldots, v_n'$ in a way such that $v_i v_i' \in E$ for all $i$ and $v_i v_j, v_i' v_j' \in E$ for all $i, j$ such that $|i-j|=1$. Now, consider the sequence of attacks on the vertices
$v_1, v_2', v_4, v_5', v_7, v_8' \ldots$ in this particular order. More formally, for any $k \geq 0$, at time $t=2k+1$, the attacked vertex is $v_{3k+1}$ and at time $t=2k+2$, the attacked vertex is $v_{3k+2}'$. Observe that, for any $t_1 \geq 1$ and $t_2 > t_1$, the attacked vertex $u_2$ at time $t_2$ is at distance at least $t_2-t_1+1$ from the attacked vertex $u_1$ at time $t_1$. Since the weight on $u_1$ can be distributed only to the vertices at distance at most $t_2-t_1$ from $u_1$ during the $t_2$-th attack, it follows that new set of weights (not coming from $u_1$) must be moved to $u_2$. Consequently, for any $k \geq 0$, the sum of the weights on the vertices of the set $\{v_i: i \leq 3k+1\} \cup \{v_i': i \leq 3k+1\}$ is at least $t=2k+1$ after a response to the $t$-th attack if $t=2k+1$ and the sum of the weights on the vertices of the set $\{v_i: i \leq 3k+2\} \cup \{v_i': i \leq 3k+2\}$ is at least $t=2k+2$ after a response to the $t$-th attack if $t=2k+2$. Thus, the inequality follows.
\end{proof}

\begin{thm}
$\fed(P_m \Box P_n) \leq \dfrac{mn}{5}+\dfrac{2(m+n)}{15}+\dfrac{39}{15}$ for any $m, n \geq 2$.
\end{thm}

\begin{proof}
We begin by placing a weight of $\frac{1}{5}$ on each vertex in the inner $P_{m-2} \Box P_{n-2}$, a weight of $\frac{7}{15}$ on each of the corner vertices and a weight of $\frac{4}{15}$ on each of the non-corner boundary vertices.  We place an additional weight of $\frac{12}{15}$ on an arbitrary vertex in the inner $P_{m-2} \Box P_{n-2}$, an additional weight of $\frac{8}{15}$ on one of the corner vertices and an additional weight of $\frac{12}{15}$ on a random vertex in the boundary (not the corner). If a vertex in the inner $P_{m-2} \Box P_{n-2}$ is attacked, the additional weight of $\frac{12}{15}$ can be distributed along $4$ disjoint paths to that vertex. If a vertex in the corner is attacked, the additional weight of $\frac{8}{15}$ can be distributed along $2$ disjoint paths to that vertex. If a vertex on the boundary is attacked, the additional weight of $\frac{11}{15}$ can be distributed along $3$ disjoint paths to that vertex. This strategy shows that a total weight of $\frac{(m-2)(n-2)}{5}+\frac{12}{15}+\frac{8(m-2)+8(n-2)}{15}+\frac{11}{15}+\frac{28}{15}+\frac{8}{15}=\frac{mn}{5}+\frac{2(m+n)}{15}+\frac{39}{15}$ can defend the graph from any sequence of attacks.
\end{proof}

\begin{thm}
$\fed(P_n \boxtimes P_m) \leq \dfrac{mn}{9}+\dfrac{16(m+n)}{9}+\dfrac{114}{9}$ for any $m, n \geq 0$.
\end{thm}

\begin{proof}
Observe that there exists $8$ disjoint paths joining any pair of vertices in the inner $P_{n-8} \boxtimes P_{m-8}$ subgrid. So, we start by placing a weight of $\frac{1}{9}$ on each of the vertices of that subgrid, then a weight of $\frac{3}{9}$ on the remaining vertices of the graph. Now we can place an additional weight of $\frac{8}{9}$ on a random vertex in the subgrid and an additional weight of $\frac{6}{9}$ on a random vertex not in the subgrid. If a vertex in the subgrid is attacked, then the additional weight of $\frac{8}{9}$ can be distributed along eight disjoint paths to that vertex, otherwise, the additional weight of $\frac{6}{9}$ can be distributed along disjoint paths leading to that vertex. This strategy shows that a total weight of $\frac{(m-8)(n-8)}{9}+\frac{8}{9}+\frac{24m+24n-192}{9}+\frac{6}{9}=\frac{mn}{9}+\frac{16(m+n)}{9}+\frac{114}{9}$ can defend the graph from any sequence of attacks.
\end{proof}


\section{Conclusion}\label{sec:conclusion}

We conclude with some open problems for future research.  In Section \ref{sec:behaviour} we noted that, for any rational number $r \geq 2$, there is a graph whose fractional eternal domination number is exactly $r$.  However, it is not clear whether or not $\fed(G)$ is necessarily rational for every finite graph $G$.

\begin{prob}\label{ques:rational}
Is $\fed(G)$ rational for every graph $G$?
\end{prob}

In Section \ref{sec:linprog}, we showed that, if a graph $G$ can be eternally fractionally dominated by $n$ {f.d}-functions, then $\fed(G)$ and an optimal guarding strategy can be computed efficiently and $\fed(G)$ is necessarily rational.  This prompts to ask the following:

\begin{prob}\label{ques:finite}
Can every finite graph $G$ be eternally fractionally dominated by $n$ {f.d.}-functions?
\end{prob}

Clearly, a positive answer to Problem \ref{ques:finite} implies a positive answer to Problem \ref{ques:rational}.

In light of the upper bound from Theorem \ref{conn} on $\fed(G)$ in terms of the connectivity and order of $G$, and of the bound given in Lemma \ref{degreebounds} on $\gamma_f(G)$, we offer the following problem:

\begin{prob}\label{ques:conn}
Does there exist some function $f$ such that every graph with connectivity $\kappa$ satisfies $\fed(G) \leq \gamma_f(G) + f(\kappa)$?
\end{prob}

Following our study of various graph classes, we are left with a number of unanswered questions.

\begin{prob}
Determine the exact value of $\fed(G)$ for all Kneser graphs.
\end{prob}

Recall that Theorem \ref{thm:petersen} determined the exact value for $n=5$ and $k=2$, and Theorem \ref{thm:kneser} gives bounds for $KG_{n,2}$.

In Section \ref{sec:Cayley}, it was proved that, if $G$ is a cubic abelian Cayley graph, then $\fed(G) < \gamma(G)$ if and only if $G$ is isomorphic to either $C_{4k+2} \Box K_2$ or $\Cay(\Z_{8k}, \{\pm 1, 4k\})$ for some positive integer $k$. For all other cubic abelian Cayley graphs, and for $C_{10} \Box K_2$, the exact value of $\fed(G)$ was computed exactly.

\begin{prob}
Let $k$ be a positive integer.  Determine the exact value of $\fed(G)$ when $G$ is isomorphic to
    \begin{compactitem}
    \item $C_{4k+2} \Box K_2 \textup{ for $k \neq 2$}$;
    \item $\Cay(\Z_{8k}, \{\pm 1, 4k\})$.
    \end{compactitem}
\end{prob}

Finally, turning to graph products, the exact value of $\fed(G)$ remains unsolved for hypercubes, Cartesian grids (except for ``ladders''), and strong grids.  Domination parameters are notoriously difficult to compute in graph products, however the results of Section \ref{sec:products} provide some initial bounds from which to work.







\subsection*{Acknowledgements}

D.~Khatri and N.~Kumar completed portions of this work while research interns at the Universit\'e de Montr\'eal, and thank the MITACS Globalink Program for its support.  A.~Krim-Yee and A.~Xu completed portions of this work while research interns at Dawson College, and thank FRQNT for their financial support.  B.~Seamone was supported by FRQNT grant number 2018-CO-210238, V.~Virgile was supported by funding from NSERC, and G.~MacGillivray was supported by NSERC.

\newpage
\bibliographystyle{siam}
\bibliography{references}


\newpage
\section*{Appendix}

\begin{figure}[h!]
    \centering
    \begin{tikzpicture}[scale=.7]
    	\begin{pgfonlayer}{nodelayer}	
    		\node [style=vertex, label=above:\color{blue}{$1$}] (0) at (-5, 10) {$v_0$};
    		\node [style=vertex, label=left:\color{blue}{$0$}] (1) at (-8, 6) {$v_1$};
    		\node [style=vertex, label=below:\color{blue}{$\frac{2}{3}$}] (2) at (-5, 2) {$v_2$};
    		\node [style=vertex, label=right:\color{blue}{$0$}] (3) at (-2, 6) {$v_3$};
    		\node [style=vertex, label=below:\color{blue}{$\frac{1}{3}$}] (4) at (-5, 8) {$v_4$};
    		\node [style=vertex, label=right:\color{blue}{$0$}] (5) at (-6.5, 6) {$v_5$};
    		\node [style=vertex, label=above:\color{blue}{$\frac{2}{3}$}] (6) at (-5, 4) {$v_6$};
    		\node [style=vertex, label=left:\color{blue}{$0$}] (7) at (-3.5, 6) {$v_7$};

    		\node [style=vertex, label=above:\color{blue}{$\frac{1}{3}$}] (32) at (5, 10) {$v_0$};
    		\node [style=vertex, label=left:\color{blue}{$0$}] (33) at (2, 6) {$v_1$};
    		\node [style=vertex, label=below:\color{blue}{$\frac{2}{3}$}] (34) at (5, 2) {$v_2$};
    		\node [style=vertex, label=right:\color{blue}{$0$}] (35) at (8, 6) {$v_3$};
    		\node [style=vertex, label=below:\color{blue}{$1$}] (36) at (5, 8) {$v_4$};
    		\node [style=vertex, label=right:\color{blue}{$0$}] (37) at (3.5, 6) {$v_5$};
    		\node [style=vertex, label=above:\color{blue}{$\frac{2}{3}$}] (38) at (5, 4) {$v_6$};
    		\node [style=vertex, label=left:\color{blue}{$0$}] (39) at (6.5, 6) {$v_7$};
    		\node [draw,align=left] at (7,9) {\tiny $v_0 \xrightarrow[]{2/3} v_4$};

    		\node [style=vertex, label=above:\color{blue}{$0$}] (40) at (-5, -2) {$v_0$};
    		\node [style=vertex, label=left:\color{blue}{$1$}] (41) at (-8, -6) {$v_1$};
    		\node [style=vertex, label=below:\color{blue}{$0$}] (42) at (-5, -10) {$v_2$};
    		\node [style=vertex, label=right:\color{blue}{$\frac{2}{3}$}] (43) at (-2, -6) {$v_3$};
    		\node [style=vertex, label=below:\color{blue}{$0$}] (44) at (-5, -4) {$v_4$};
    		\node [style=vertex, label=right:\color{blue}{$\frac{1}{3}$}] (45) at (-6.5, -6) {$v_5$};
    		\node [style=vertex, label=above:\color{blue}{$0$}] (46) at (-5, -8) {$v_6$};
    		\node [style=vertex, label=left:\color{blue}{$\frac{2}{3}$}] (47) at (-3.5, -6) {$v_7$};
    		\node [draw,align=left] at (-3,-3) {\tiny $v_0 \xrightarrow[]{1} v_1$\\ \tiny $v_4 \xrightarrow[]{1/3} v_5$\\ \tiny $v_6 \xrightarrow[]{2/3} v_7$};

    		\node [style=vertex, label=above:\color{blue}{$0$}] (48) at (5, -2) {$v_0$};
    		\node [style=vertex, label=left:\color{blue}{$\frac{2}{3}$}] (49) at (2, -6) {$v_1$};
    		\node [style=vertex, label=below:\color{blue}{$0$}] (50) at (5, -10) {$v_2$};
    		\node [style=vertex, label=right:\color{blue}{$\frac{1}{3}$}] (51) at (8, -6) {$v_3$};
    		\node [style=vertex, label=below:\color{blue}{$0$}] (52) at (5, -4) {$v_4$};
    		\node [style=vertex, label=right:\color{blue}{$\frac{2}{3}$}] (53) at (3.5, -6) {$v_5$};
    		\node [style=vertex, label=above:\color{blue}{$0$}] (54) at (5, -8) {$v_6$};
    		\node [style=vertex, label=left:\color{blue}{$1$}] (55) at (6.5, -6) {$v_7$};
    		\node [draw,align=left] at (7,-3) {\tiny $v_0 \xrightarrow[]{1} v_7$\\ \tiny $v_2 \xrightarrow[]{2/3} v_1$\\ \tiny $v_6 \xrightarrow[]{2/3} v_5$};
    	\end{pgfonlayer}
    	\begin{pgfonlayer}{edgelayer}
    		\draw (0) to (1);
    		\draw (1) to (2);
    		\draw (2) to (3);
    		\draw (3) to (4);
    		\draw (4) to (5);
    		\draw (5) to (6);
    		\draw (6) to (7);
    		\draw (7) to (0);
    		\draw (0) to (4);
    		\draw (1) to (5);
    		\draw (2) to (6);
    		\draw (3) to (7);
    		\draw (32) to (33);
    		\draw (33) to (34);
    		\draw (34) to (35);
    		\draw (35) to (36);
    		\draw (36) to (37);
    		\draw (37) to (38);
    		\draw (38) to (39);
    		\draw (39) to (32);
    		\draw (32) to (36);
    		\draw (33) to (37);
    		\draw (34) to (38);
    		\draw (35) to (39);
    		\draw (40) to (41);
    		\draw (41) to (42);
    		\draw (42) to (43);
    		\draw (43) to (44);
    		\draw (44) to (45);
    		\draw (45) to (46);
    		\draw (46) to (47);
    		\draw (47) to (40);
    		\draw (40) to (44);
    		\draw (41) to (45);
    		\draw (42) to (46);
    		\draw (43) to (47);
    		\draw (48) to (49);
    		\draw (49) to (50);
    		\draw (50) to (51);
    		\draw (51) to (52);
    		\draw (52) to (53);
    		\draw (53) to (54);
    		\draw (54) to (55);
    		\draw (55) to (48);
    		\draw (48) to (52);
    		\draw (49) to (53);
    		\draw (50) to (54);
    		\draw (51) to (55);
    	\end{pgfonlayer}
    \end{tikzpicture}
    
    \caption{Configuration of the guards in $Cay(\mathbb{Z}_8, \{\pm 1, 4\})$ after an attack on the vertices $v_0, v_1, v_4, v_7$. The graph on the top left corresponds to the initial configuration of the guards.}
    \label{Figure:CayleyZ8}
\end{figure}


\begin{figure}[h]
    \centering
    \begin{tikzpicture}[scale=.8]
    	\begin{pgfonlayer}{nodelayer}	
    		\node [style=vertex, label=above:\color{blue}{$0$}] (0) at (-5, 10) {$v_0$};
    		\node [style=vertex, label=left:\color{blue}{$\frac{1}{3}$}] (1) at (-8, 6) {$v_1$};
    		\node [style=vertex, label=below:\color{blue}{$0$}] (2) at (-5, 2) {$v_2$};
    		\node [style=vertex, label=right:\color{blue}{$\frac{2}{3}$}] (3) at (-2, 6) {$v_3$};
    		\node [style=vertex, label=below:\color{blue}{$0$}] (4) at (-5, 8) {$v_4$};
    		\node [style=vertex, label=right:\color{blue}{$1$}] (5) at (-6.5, 6) {$v_5$};
    		\node [style=vertex, label=above:\color{blue}{$0$}] (6) at (-5, 4) {$v_6$};
    		\node [style=vertex, label=left:\color{blue}{$\frac{2}{3}$}] (7) at (-3.5, 6) {$v_7$};
    		\node [draw,align=left] at (-3,9) {\tiny $v_0 \xrightarrow[]{1/3} v_1$\\ \tiny $v_0 \xrightarrow[]{2/3} v_7$\\ \tiny $v_4 \xrightarrow[]{1/3} v_5$\\ \tiny $v_6 \xrightarrow[]{2/3} v_5$\\ \tiny $v_2 \xrightarrow[]{2/3} v_3$};

    		\node [style=vertex, label=above:\color{blue}{$\frac{2}{3}$}] (32) at (5, 10) {$v_0$};
    		\node [style=vertex, label=left:\color{blue}{$0$}] (33) at (2, 6) {$v_1$};
    		\node [style=vertex, label=below:\color{blue}{$\frac{1}{3}$}] (34) at (5, 2) {$v_2$};
    		\node [style=vertex, label=right:\color{blue}{$0$}] (35) at (8, 6) {$v_3$};
    		\node [style=vertex, label=below:\color{blue}{$\frac{2}{3}$}] (36) at (5, 8) {$v_4$};
    		\node [style=vertex, label=right:\color{blue}{$0$}] (37) at (3.5, 6) {$v_5$};
    		\node [style=vertex, label=above:\color{blue}{$1$}] (38) at (5, 4) {$v_6$};
    		\node [style=vertex, label=left:\color{blue}{$0$}] (39) at (6.5, 6) {$v_7$};
    		\node [draw,align=left] at (7,9) {\tiny $v_0 \xrightarrow[]{1/3} v_4$\\ \tiny $v_2 \xrightarrow[]{1/3} v_6$};

    		\node [style=vertex, label=above:\color{blue}{$\frac{2}{3}$}] (40) at (-5, -2) {$v_0$};
    		\node [style=vertex, label=left:\color{blue}{$0$}] (41) at (-8, -6) {$v_1$};
    		\node [style=vertex, label=below:\color{blue}{$1$}] (42) at (-5, -10) {$v_2$};
    		\node [style=vertex, label=right:\color{blue}{$0$}] (43) at (-2, -6) {$v_3$};
    		\node [style=vertex, label=below:\color{blue}{$\frac{2}{3}$}] (44) at (-5, -4) {$v_4$};
    		\node [style=vertex, label=right:\color{blue}{$0$}] (45) at (-6.5, -6) {$v_5$};
    		\node [style=vertex, label=above:\color{blue}{$\frac{1}{3}$}] (46) at (-5, -8) {$v_6$};
    		\node [style=vertex, label=left:\color{blue}{$0$}] (47) at (-3.5, -6) {$v_7$};
    		\node [draw,align=left] at (-3,-3) {\tiny $v_0 \xrightarrow[]{1/3} v_4$\\ \tiny $v_6 \xrightarrow[]{1/3} v_2$};

    		\node [style=vertex, label=above:\color{blue}{$0$}] (48) at (5, -2) {$v_0$};
    		\node [style=vertex, label=left:\color{blue}{$\frac{2}{3}$}] (49) at (2, -6) {$v_1$};
    		\node [style=vertex, label=below:\color{blue}{$0$}] (50) at (5, -10) {$v_2$};
    		\node [style=vertex, label=right:\color{blue}{$1$}] (51) at (8, -6) {$v_3$};
    		\node [style=vertex, label=below:\color{blue}{$0$}] (52) at (5, -4) {$v_4$};
    		\node [style=vertex, label=right:\color{blue}{$\frac{2}{3}$}] (53) at (3.5, -6) {$v_5$};
    		\node [style=vertex, label=above:\color{blue}{$0$}] (54) at (5, -8) {$v_6$};
    		\node [style=vertex, label=left:\color{blue}{$\frac{1}{3}$}] (55) at (6.5, -6) {$v_7$};
    		\node [draw,align=left] at (7,-3) {\tiny $v_0 \xrightarrow[]{2/3} v_1$\\ \tiny $v_0 \xrightarrow[]{1/3} v_7$\\ \tiny $v_6 \xrightarrow[]{2/3} v_5$};
    	\end{pgfonlayer}
    	\begin{pgfonlayer}{edgelayer}
    		\draw (0) to (1);
    		\draw (1) to (2);
    		\draw (2) to (3);
    		\draw (3) to (4);
    		\draw (4) to (5);
    		\draw (5) to (6);
    		\draw (6) to (7);
    		\draw (7) to (0);
    		\draw (0) to (4);
    		\draw (1) to (5);
    		\draw (2) to (6);
    		\draw (3) to (7);
    		\draw (32) to (33);
    		\draw (33) to (34);
    		\draw (34) to (35);
    		\draw (35) to (36);
    		\draw (36) to (37);
    		\draw (37) to (38);
    		\draw (38) to (39);
    		\draw (39) to (32);
    		\draw (32) to (36);
    		\draw (33) to (37);
    		\draw (34) to (38);
    		\draw (35) to (39);
    		\draw (40) to (41);
    		\draw (41) to (42);
    		\draw (42) to (43);
    		\draw (43) to (44);
    		\draw (44) to (45);
    		\draw (45) to (46);
    		\draw (46) to (47);
    		\draw (47) to (40);
    		\draw (40) to (44);
    		\draw (41) to (45);
    		\draw (42) to (46);
    		\draw (43) to (47);
    		\draw (48) to (49);
    		\draw (49) to (50);
    		\draw (50) to (51);
    		\draw (51) to (52);
    		\draw (52) to (53);
    		\draw (53) to (54);
    		\draw (54) to (55);
    		\draw (55) to (48);
    		\draw (48) to (52);
    		\draw (49) to (53);
    		\draw (50) to (54);
    		\draw (51) to (55);
    	\end{pgfonlayer}
    \end{tikzpicture}

    \caption{Configuration of the guards in $Cay(\mathbb{Z}_8, \{\pm 1, 4\})$ after an attack on the vertices $v_2, v_3, v_5, v_6$. The graph on the top left corresponds to the initial configuration of the guards.}
    \label{Figure:CayleyZ8-2}
\end{figure}


\newpage
\begin{figure}[h]
    \centering
	\begin{tikzpicture}[scale=.7]
		\begin{pgfonlayer}{nodelayer}
			\node [style=vertex, label=left:\color{blue}{\tiny $1$}] (0) at (-8.2, 11.2) {$v_0$};
			\node [style=vertex, label=above:\color{blue}{\tiny $0$}] (1) at (-7, 10) {$v_1$};
			\node [style=vertex, label=above:\color{blue}{\tiny $1/5$}] (2) at (-5, 10) {$v_2$};
			\node [style=vertex, label=above:\color{blue}{\tiny $2/5$}] (3) at (-3, 10) {$v_3$};
			\node [style=vertex, label=above:\color{blue}{\tiny $1/5$}] (4) at (-1, 10) {$v_4$};
			\node [style=vertex, label=above:\color{blue}{\tiny $1/5$}] (5) at (1, 10) {$v_5$};
			\node [style=vertex, label=above:\color{blue}{\tiny $1/5$}] (6) at (3, 10) {$v_6$};
			\node [style=vertex, label=above:\color{blue}{\tiny $2/5$}] (7) at (5, 10) {$v_7$};
			\node [style=vertex, label=above:\color{blue}{\tiny $1/5$}] (8) at (7, 10) {$v_8$};
			\node [style=vertex, label=right:\color{blue}{\tiny $0$}] (9) at (8.2, 11.2) {$v_9$};
			\node [style=vertex, label=left:\color{blue}{\tiny $2/5$}] (10) at (-8.2, 6.8) {$u_0$};
			\node [style=vertex, label=below:\color{blue}{\tiny $1/5$}] (11) at (-7, 8) {$u_1$};
			\node [style=vertex, label=below:\color{blue}{\tiny $2/5$}] (12) at (-5, 8) {$u_2$};
			\node [style=vertex, label=below:\color{blue}{\tiny $1/5$}] (13) at (-3, 8) {$u_3$};
			\node [style=vertex, label=below:\color{blue}{\tiny $1/5$}] (14) at (-1, 8) {$u_4$};
			\node [style=vertex, label=below:\color{blue}{\tiny $2/5$}] (15) at (1, 8) {$u_5$};
			\node [style=vertex, label=below:\color{blue}{\tiny $1/5$}] (16) at (3, 8) {$u_6$};
			\node [style=vertex, label=below:\color{blue}{\tiny $1/5$}] (17) at (5, 8) {$u_7$};
			\node [style=vertex, label=below:\color{blue}{\tiny $2/5$}] (18) at (7, 8) {$u_8$};
			\node [style=vertex, label=right:\color{blue}{\tiny $1/5$}] (19) at (8.2, 6.8) {$u_9$};

			\node [style=vertex, label=left:\color{blue}{\tiny $0$}] (20) at (-8.2, 5.2) {$v_0$};
			\node [style=vertex, label=above:\color{blue}{\tiny $1$}] (21) at (-7, 4) {$v_1$};
			\node [style=vertex, label=above:\color{blue}{\tiny $0$}] (22) at (-5, 4) {$v_2$};
			\node [style=vertex, label=above:\color{blue}{\tiny $1/5$}] (23) at (-3, 4) {$v_3$};
			\node [style=vertex, label=above:\color{blue}{\tiny $2/5$}] (24) at (-1, 4) {$v_4$};
			\node [style=vertex, label=above:\color{blue}{\tiny $1/5$}] (25) at (1, 4) {$v_5$};
			\node [style=vertex, label=above:\color{blue}{\tiny $1/5$}] (26) at (3, 4) {$v_6$};
			\node [style=vertex, label=above:\color{blue}{\tiny $1/5$}] (27) at (5, 4) {$v_7$};
			\node [style=vertex, label=above:\color{blue}{\tiny $2/5$}] (28) at (7, 4) {$v_8$};
			\node [style=vertex, label=right:\color{blue}{\tiny $1/5$}] (29) at (8.2, 5.2) {$v_9$};
			\node [style=vertex, label=left:\color{blue}{\tiny $1/5$}] (30) at (-8.2, 0.8) {$u_0$};
			\node [style=vertex, label=below:\color{blue}{\tiny $2/5$}] (31) at (-7, 2) {$u_1$};
			\node [style=vertex, label=below:\color{blue}{\tiny $1/5$}] (32) at (-5, 2) {$u_2$};
			\node [style=vertex, label=below:\color{blue}{\tiny $2/5$}] (33) at (-3, 2) {$u_3$};
			\node [style=vertex, label=below:\color{blue}{\tiny $1/5$}] (34) at (-1, 2) {$u_4$};
			\node [style=vertex, label=below:\color{blue}{\tiny $1/5$}] (35) at (1, 2) {$u_5$};
			\node [style=vertex, label=below:\color{blue}{\tiny $2/5$}] (36) at (3, 2) {$u_6$};
			\node [style=vertex, label=below:\color{blue}{\tiny $1/5$}] (37) at (5, 2) {$u_7$};
			\node [style=vertex, label=below:\color{blue}{\tiny $1/5$}] (38) at (7, 2) {$u_8$};
			\node [style=vertex, label=right:\color{blue}{\tiny $2/5$}] (39) at (8.2, 0.8) {$u_9$};

			\node [style=vertex, label=left:\color{blue}{\tiny $1/5$}] (40) at (-8.2, -0.8) {$v_0$};
			\node [style=vertex, label=above:\color{blue}{\tiny $0$}] (41) at (-7, -2) {$v_1$};
			\node [style=vertex, label=above:\color{blue}{\tiny $1$}] (42) at (-5, -2) {$v_2$};
			\node [style=vertex, label=above:\color{blue}{\tiny $1/5$}] (43) at (-3, -2) {$v_3$};
			\node [style=vertex, label=above:\color{blue}{\tiny $2/5$}] (44) at (-1, -2) {$v_4$};
			\node [style=vertex, label=above:\color{blue}{\tiny $1/5$}] (45) at (1, -2) {$v_5$};
			\node [style=vertex, label=above:\color{blue}{\tiny $1/5$}] (46) at (3, -2) {$v_6$};
			\node [style=vertex, label=above:\color{blue}{\tiny $1/5$}] (47) at (5, -2) {$v_7$};
			\node [style=vertex, label=above:\color{blue}{\tiny $2/5$}] (48) at (7, -2) {$v_8$};
			\node [style=vertex, label=right:\color{blue}{\tiny $1/5$}] (49) at (8.2, -0.8) {$v_9$};
			\node [style=vertex, label=left:\color{blue}{\tiny $2/5$}] (50) at (-8.2, -5.2) {$u_0$};
			\node [style=vertex, label=below:\color{blue}{\tiny $1/5$}] (51) at (-7, -4) {$u_1$};
			\node [style=vertex, label=below:\color{blue}{\tiny $2/5$}] (52) at (-5, -4) {$u_2$};
			\node [style=vertex, label=below:\color{blue}{\tiny $1/5$}] (53) at (-3, -4) {$u_3$};
			\node [style=vertex, label=below:\color{blue}{\tiny $2/5$}] (54) at (-1, -4) {$u_4$};
			\node [style=vertex, label=below:\color{blue}{\tiny $1/5$}] (55) at (1, -4) {$u_5$};
			\node [style=vertex, label=below:\color{blue}{\tiny $1/5$}] (56) at (3, -4) {$u_6$};
			\node [style=vertex, label=below:\color{blue}{\tiny $2/5$}] (57) at (5, -4) {$u_7$};
			\node [style=vertex, label=below:\color{blue}{\tiny $1/5$}] (58) at (7, -4) {$u_8$};
			\node [style=vertex, label=right:\color{blue}{\tiny $1/5$}] (59) at (8.2, -5.2) {$u_9$};

			\node [style=vertex, label=left:\color{blue}{\tiny $2/5$}] (60) at (-8.2, -6.8) {$v_0$};
			\node [style=vertex, label=above:\color{blue}{\tiny $1/5$}] (61) at (-7, -8) {$v_1$};
			\node [style=vertex, label=above:\color{blue}{\tiny $0$}] (62) at (-5, -8) {$v_2$};
			\node [style=vertex, label=above:\color{blue}{\tiny $1$}] (63) at (-3, -8) {$v_3$};
			\node [style=vertex, label=above:\color{blue}{\tiny $0$}] (64) at (-1, -8) {$v_4$};
			\node [style=vertex, label=above:\color{blue}{\tiny $1/5$}] (65) at (1, -8) {$v_5$};
			\node [style=vertex, label=above:\color{blue}{\tiny $2/5$}] (66) at (3, -8) {$v_6$};
			\node [style=vertex, label=above:\color{blue}{\tiny $1/5$}] (67) at (5, -8) {$v_7$};
			\node [style=vertex, label=above:\color{blue}{\tiny $1/5$}] (68) at (7, -8) {$v_8$};
			\node [style=vertex, label=right:\color{blue}{\tiny $1/5$}] (69) at (8.2, -6.8) {$v_9$};
			\node [style=vertex, label=left:\color{blue}{\tiny $1/5$}] (70) at (-8.2, -11.2) {$u_0$};
			\node [style=vertex, label=below:\color{blue}{\tiny $2/5$}] (71) at (-7, -10) {$u_1$};
			\node [style=vertex, label=below:\color{blue}{\tiny $1/5$}] (72) at (-5, -10) {$u_2$};
			\node [style=vertex, label=below:\color{blue}{\tiny $2/5$}] (73) at (-3, -10) {$u_3$};
			\node [style=vertex, label=below:\color{blue}{\tiny $1/5$}] (74) at (-1, -10) {$u_4$};
			\node [style=vertex, label=below:\color{blue}{\tiny $2/5$}] (75) at (1, -10) {$u_5$};
			\node [style=vertex, label=below:\color{blue}{\tiny $1/5$}] (76) at (3, -10) {$u_6$};
			\node [style=vertex, label=below:\color{blue}{\tiny $1/5$}] (77) at (5, -10) {$u_7$};
			\node [style=vertex, label=below:\color{blue}{\tiny $2/5$}] (78) at (7, -10) {$u_8$};
			\node [style=vertex, label=right:\color{blue}{\tiny $1/5$}] (79) at (8.2, -11.2) {$u_9$};
		\end{pgfonlayer}
		\begin{pgfonlayer}{edgelayer}
			\draw (0) to (1);
			\draw (1) to (2);
			\draw (2) to (3);
			\draw (3) to (4);
			\draw (4) to (5);
			\draw (5) to (6);
			\draw (6) to (7);
			\draw (7) to (8);
			\draw (8) to (9);
			\draw (9) to (0);
			\draw (10) to (11);
			\draw (11) to (12);
			\draw (12) to (13);
			\draw (13) to (14);
			\draw (14) to (15);
			\draw (15) to (16);
			\draw (16) to (17);
			\draw (17) to (18);
			\draw (18) to (19);
			\draw (19) to (10);
			\draw (0) to (10);
			\draw (1) to (11);
			\draw (2) to (12);
			\draw (3) to (13);
			\draw (4) to (14);
			\draw (5) to (15);
			\draw (6) to (16);
			\draw (7) to (17);
			\draw (8) to (18);
			\draw (9) to (19);
			\draw (20) to (21);
			\draw (21) to (22);
			\draw (22) to (23);
			\draw (23) to (24);
			\draw (24) to (25);
			\draw (25) to (26);
			\draw (26) to (27);
			\draw (27) to (28);
			\draw (28) to (29);
			\draw (29) to (20);
			\draw (30) to (31);
			\draw (31) to (32);
			\draw (32) to (33);
			\draw (33) to (34);
			\draw (34) to (35);
			\draw (35) to (36);
			\draw (36) to (37);
			\draw (37) to (38);
			\draw (38) to (39);
			\draw (39) to (30);
			\draw (20) to (30);
			\draw (21) to (31);
			\draw (22) to (32);
			\draw (23) to (33);
			\draw (24) to (34);
			\draw (25) to (35);
			\draw (26) to (36);
			\draw (27) to (37);
			\draw (28) to (38);
			\draw (29) to (39);
			\draw (40) to (41);
			\draw (41) to (42);
			\draw (42) to (43);
			\draw (43) to (44);
			\draw (44) to (45);
			\draw (45) to (46);
			\draw (46) to (47);
			\draw (47) to (48);
			\draw (48) to (49);
			\draw (49) to (40);
			\draw (50) to (51);
			\draw (51) to (52);
			\draw (52) to (53);
			\draw (53) to (54);
			\draw (54) to (55);
			\draw (55) to (56);
			\draw (56) to (57);
			\draw (57) to (58);
			\draw (58) to (59);
			\draw (59) to (50);
			\draw (40) to (50);
			\draw (41) to (51);
			\draw (42) to (52);
			\draw (43) to (53);
			\draw (44) to (54);
			\draw (45) to (55);
			\draw (46) to (56);
			\draw (47) to (57);
			\draw (48) to (58);
			\draw (49) to (59);
			\draw (60) to (61);
			\draw (61) to (62);
			\draw (62) to (63);
			\draw (63) to (64);
			\draw (64) to (65);
			\draw (65) to (66);
			\draw (66) to (67);
			\draw (67) to (68);
			\draw (68) to (69);
			\draw (69) to (60);
			\draw (70) to (71);
			\draw (71) to (72);
			\draw (72) to (73);
			\draw (73) to (74);
			\draw (74) to (75);
			\draw (75) to (76);
			\draw (76) to (77);
			\draw (77) to (78);
			\draw (78) to (79);
			\draw (79) to (70);
			\draw (60) to (70);
			\draw (61) to (71);
			\draw (62) to (72);
			\draw (63) to (73);
			\draw (64) to (74);
			\draw (65) to (75);
			\draw (66) to (76);
			\draw (67) to (77);
			\draw (68) to (78);
			\draw (69) to (79);
		\end{pgfonlayer}
	\end{tikzpicture}
	
    \caption{Configuration of the guards in $C_{10} \Box K_2$ after an attack on the vertices $v_0, v_1, v_2, v_3$. The graph on top corresponds to the initial configuration of the guards.}
	\label{Figure:C10}
\end{figure}


\newpage
\begin{figure}[h]
    \centering
	\begin{tikzpicture}[scale=.7]
		\begin{pgfonlayer}{nodelayer}
			\node [style=vertex, label=left:\color{blue}{\tiny $1/5$}] (0) at (-8.2, 11.2) {$v_0$};
			\node [style=vertex, label=above:\color{blue}{\tiny $2/5$}] (1) at (-7, 10) {$v_1$};
			\node [style=vertex, label=above:\color{blue}{\tiny $1/5$}] (2) at (-5, 10) {$v_2$};
			\node [style=vertex, label=above:\color{blue}{\tiny $0$}] (3) at (-3, 10) {$v_3$};
			\node [style=vertex, label=above:\color{blue}{\tiny $1$}] (4) at (-1, 10) {$v_4$};
			\node [style=vertex, label=above:\color{blue}{\tiny $0$}] (5) at (1, 10) {$v_5$};
			\node [style=vertex, label=above:\color{blue}{\tiny $1/5$}] (6) at (3, 10) {$v_6$};
			\node [style=vertex, label=above:\color{blue}{\tiny $2/5$}] (7) at (5, 10) {$v_7$};
			\node [style=vertex, label=above:\color{blue}{\tiny $1/5$}] (8) at (7, 10) {$v_8$};
			\node [style=vertex, label=right:\color{blue}{\tiny $1/5$}] (9) at (8.2, 11.2) {$v_9$};
			\node [style=vertex, label=left:\color{blue}{\tiny $1/5$}] (10) at (-8.2, 6.8) {$u_0$};
			\node [style=vertex, label=below:\color{blue}{\tiny $1/5$}] (11) at (-7, 8) {$u_1$};
			\node [style=vertex, label=below:\color{blue}{\tiny $2/5$}] (12) at (-5, 8) {$u_2$};
			\node [style=vertex, label=below:\color{blue}{\tiny $1/5$}] (13) at (-3, 8) {$u_3$};
			\node [style=vertex, label=below:\color{blue}{\tiny $2/5$}] (14) at (-1, 8) {$u_4$};
			\node [style=vertex, label=below:\color{blue}{\tiny $1/5$}] (15) at (1, 8) {$u_5$};
			\node [style=vertex, label=below:\color{blue}{\tiny $2/5$}] (16) at (3, 8) {$u_6$};
			\node [style=vertex, label=below:\color{blue}{\tiny $1/5$}] (17) at (5, 8) {$u_7$};
			\node [style=vertex, label=below:\color{blue}{\tiny $1/5$}] (18) at (7, 8) {$u_8$};
			\node [style=vertex, label=right:\color{blue}{\tiny $2/5$}] (19) at (8.2, 6.8) {$u_9$};

			\node [style=vertex, label=left:\color{blue}{\tiny $1/5$}] (20) at (-8.2, 5.2) {$v_0$};
			\node [style=vertex, label=above:\color{blue}{\tiny $1/5$}] (21) at (-7, 4) {$v_1$};
			\node [style=vertex, label=above:\color{blue}{\tiny $2/5$}] (22) at (-5, 4) {$v_2$};
			\node [style=vertex, label=above:\color{blue}{\tiny $1/5$}] (23) at (-3, 4) {$v_3$};
			\node [style=vertex, label=above:\color{blue}{\tiny $0$}] (24) at (-1, 4) {$v_4$};
			\node [style=vertex, label=above:\color{blue}{\tiny $1$}] (25) at (1, 4) {$v_5$};
			\node [style=vertex, label=above:\color{blue}{\tiny $0$}] (26) at (3, 4) {$v_6$};
			\node [style=vertex, label=above:\color{blue}{\tiny $1/5$}] (27) at (5, 4) {$v_7$};
			\node [style=vertex, label=above:\color{blue}{\tiny $2/5$}] (28) at (7, 4) {$v_8$};
			\node [style=vertex, label=right:\color{blue}{\tiny $1/5$}] (29) at (8.2, 5.2) {$v_9$};
			\node [style=vertex, label=left:\color{blue}{\tiny $2/5$}] (30) at (-8.2, 0.8) {$u_0$};
			\node [style=vertex, label=below:\color{blue}{\tiny $1/5$}] (31) at (-7, 2) {$u_1$};
			\node [style=vertex, label=below:\color{blue}{\tiny $1/5$}] (32) at (-5, 2) {$u_2$};
			\node [style=vertex, label=below:\color{blue}{\tiny $2/5$}] (33) at (-3, 2) {$u_3$};
			\node [style=vertex, label=below:\color{blue}{\tiny $1/5$}] (34) at (-1, 2) {$u_4$};
			\node [style=vertex, label=below:\color{blue}{\tiny $2/5$}] (35) at (1, 2) {$u_5$};
			\node [style=vertex, label=below:\color{blue}{\tiny $1/5$}] (36) at (3, 2) {$u_6$};
			\node [style=vertex, label=below:\color{blue}{\tiny $2/5$}] (37) at (5, 2) {$u_7$};
			\node [style=vertex, label=below:\color{blue}{\tiny $1/5$}] (38) at (7, 2) {$u_8$};
			\node [style=vertex, label=right:\color{blue}{\tiny $1/5$}] (39) at (8.2, 0.8) {$u_9$};

			\node [style=vertex, label=left:\color{blue}{\tiny $2/5$}] (40) at (-8.2, -0.8) {$v_0$};
			\node [style=vertex, label=above:\color{blue}{\tiny $1/5$}] (41) at (-7, -2) {$v_1$};
			\node [style=vertex, label=above:\color{blue}{\tiny $2/5$}] (42) at (-5, -2) {$v_2$};
			\node [style=vertex, label=above:\color{blue}{\tiny $1/5$}] (43) at (-3, -2) {$v_3$};
			\node [style=vertex, label=above:\color{blue}{\tiny $1/5$}] (44) at (-1, -2) {$v_4$};
			\node [style=vertex, label=above:\color{blue}{\tiny $2/5$}] (45) at (1, -2) {$v_5$};
			\node [style=vertex, label=above:\color{blue}{\tiny $1/5$}] (46) at (3, -2) {$v_6$};
			\node [style=vertex, label=above:\color{blue}{\tiny $1/5$}] (47) at (5, -2) {$v_7$};
			\node [style=vertex, label=above:\color{blue}{\tiny $2/5$}] (48) at (7, -2) {$v_8$};
			\node [style=vertex, label=right:\color{blue}{\tiny $1/5$}] (49) at (8.2, -0.8) {$v_9$};
			\node [style=vertex, label=left:\color{blue}{\tiny $1$}] (50) at (-8.2, -5.2) {$u_0$};
			\node [style=vertex, label=below:\color{blue}{\tiny $0$}] (51) at (-7, -4) {$u_1$};
			\node [style=vertex, label=below:\color{blue}{\tiny $1/5$}] (52) at (-5, -4) {$u_2$};
			\node [style=vertex, label=below:\color{blue}{\tiny $2/5$}] (53) at (-3, -4) {$u_3$};
			\node [style=vertex, label=below:\color{blue}{\tiny $1/5$}] (54) at (-1, -4) {$u_4$};
			\node [style=vertex, label=below:\color{blue}{\tiny $1/5$}] (55) at (1, -4) {$u_5$};
			\node [style=vertex, label=below:\color{blue}{\tiny $1/5$}] (56) at (3, -4) {$u_6$};
			\node [style=vertex, label=below:\color{blue}{\tiny $2/5$}] (57) at (5, -4) {$u_7$};
			\node [style=vertex, label=below:\color{blue}{\tiny $1/5$}] (58) at (7, -4) {$u_8$};
			\node [style=vertex, label=right:\color{blue}{\tiny $0$}] (59) at (8.2, -5.2) {$u_9$};

			\node [style=vertex, label=left:\color{blue}{\tiny $1/5$}] (60) at (-8.2, -6.8) {$v_0$};
			\node [style=vertex, label=above:\color{blue}{\tiny $2/5$}] (61) at (-7, -8) {$v_1$};
			\node [style=vertex, label=above:\color{blue}{\tiny $1/5$}] (62) at (-5, -8) {$v_2$};
			\node [style=vertex, label=above:\color{blue}{\tiny $2/5$}] (63) at (-3, -8) {$v_3$};
			\node [style=vertex, label=above:\color{blue}{\tiny $1/5$}] (64) at (-1, -8) {$v_4$};
			\node [style=vertex, label=above:\color{blue}{\tiny $1/5$}] (65) at (1, -8) {$v_5$};
			\node [style=vertex, label=above:\color{blue}{\tiny $2/5$}] (66) at (3, -8) {$v_6$};
			\node [style=vertex, label=above:\color{blue}{\tiny $1/5$}] (67) at (5, -8) {$v_7$};
			\node [style=vertex, label=above:\color{blue}{\tiny $1/5$}] (68) at (7, -8) {$v_8$};
			\node [style=vertex, label=right:\color{blue}{\tiny $2/5$}] (69) at (8.2, -6.8) {$v_9$};
			\node [style=vertex, label=left:\color{blue}{\tiny $0$}] (70) at (-8.2, -11.2) {$u_0$};
			\node [style=vertex, label=below:\color{blue}{\tiny $1$}] (71) at (-7, -10) {$u_1$};
			\node [style=vertex, label=below:\color{blue}{\tiny $0$}] (72) at (-5, -10) {$u_2$};
			\node [style=vertex, label=below:\color{blue}{\tiny $1/5$}] (73) at (-3, -10) {$u_3$};
			\node [style=vertex, label=below:\color{blue}{\tiny $2/5$}] (74) at (-1, -10) {$u_4$};
			\node [style=vertex, label=below:\color{blue}{\tiny $1/5$}] (75) at (1, -10) {$u_5$};
			\node [style=vertex, label=below:\color{blue}{\tiny $1/5$}] (76) at (3, -10) {$u_6$};
			\node [style=vertex, label=below:\color{blue}{\tiny $1/5$}] (77) at (5, -10) {$u_7$};
			\node [style=vertex, label=below:\color{blue}{\tiny $2/5$}] (78) at (7, -10) {$u_8$};
			\node [style=vertex, label=right:\color{blue}{\tiny $1/5$}] (79) at (8.2, -11.2) {$u_9$};
		\end{pgfonlayer}
		\begin{pgfonlayer}{edgelayer}
			\draw (0) to (1);
			\draw (1) to (2);
			\draw (2) to (3);
			\draw (3) to (4);
			\draw (4) to (5);
			\draw (5) to (6);
			\draw (6) to (7);
			\draw (7) to (8);
			\draw (8) to (9);
			\draw (9) to (0);
			\draw (10) to (11);
			\draw (11) to (12);
			\draw (12) to (13);
			\draw (13) to (14);
			\draw (14) to (15);
			\draw (15) to (16);
			\draw (16) to (17);
			\draw (17) to (18);
			\draw (18) to (19);
			\draw (19) to (10);
			\draw (0) to (10);
			\draw (1) to (11);
			\draw (2) to (12);
			\draw (3) to (13);
			\draw (4) to (14);
			\draw (5) to (15);
			\draw (6) to (16);
			\draw (7) to (17);
			\draw (8) to (18);
			\draw (9) to (19);
			\draw (20) to (21);
			\draw (21) to (22);
			\draw (22) to (23);
			\draw (23) to (24);
			\draw (24) to (25);
			\draw (25) to (26);
			\draw (26) to (27);
			\draw (27) to (28);
			\draw (28) to (29);
			\draw (29) to (20);
			\draw (30) to (31);
			\draw (31) to (32);
			\draw (32) to (33);
			\draw (33) to (34);
			\draw (34) to (35);
			\draw (35) to (36);
			\draw (36) to (37);
			\draw (37) to (38);
			\draw (38) to (39);
			\draw (39) to (30);
			\draw (20) to (30);
			\draw (21) to (31);
			\draw (22) to (32);
			\draw (23) to (33);
			\draw (24) to (34);
			\draw (25) to (35);
			\draw (26) to (36);
			\draw (27) to (37);
			\draw (28) to (38);
			\draw (29) to (39);
			\draw (40) to (41);
			\draw (41) to (42);
			\draw (42) to (43);
			\draw (43) to (44);
			\draw (44) to (45);
			\draw (45) to (46);
			\draw (46) to (47);
			\draw (47) to (48);
			\draw (48) to (49);
			\draw (49) to (40);
			\draw (50) to (51);
			\draw (51) to (52);
			\draw (52) to (53);
			\draw (53) to (54);
			\draw (54) to (55);
			\draw (55) to (56);
			\draw (56) to (57);
			\draw (57) to (58);
			\draw (58) to (59);
			\draw (59) to (50);
			\draw (40) to (50);
			\draw (41) to (51);
			\draw (42) to (52);
			\draw (43) to (53);
			\draw (44) to (54);
			\draw (45) to (55);
			\draw (46) to (56);
			\draw (47) to (57);
			\draw (48) to (58);
			\draw (49) to (59);
			\draw (60) to (61);
			\draw (61) to (62);
			\draw (62) to (63);
			\draw (63) to (64);
			\draw (64) to (65);
			\draw (65) to (66);
			\draw (66) to (67);
			\draw (67) to (68);
			\draw (68) to (69);
			\draw (69) to (60);
			\draw (70) to (71);
			\draw (71) to (72);
			\draw (72) to (73);
			\draw (73) to (74);
			\draw (74) to (75);
			\draw (75) to (76);
			\draw (76) to (77);
			\draw (77) to (78);
			\draw (78) to (79);
			\draw (79) to (70);
			\draw (60) to (70);
			\draw (61) to (71);
			\draw (62) to (72);
			\draw (63) to (73);
			\draw (64) to (74);
			\draw (65) to (75);
			\draw (66) to (76);
			\draw (67) to (77);
			\draw (68) to (78);
			\draw (69) to (79);
		\end{pgfonlayer}
	\end{tikzpicture}
	
    \caption{Configuration of the guards in $C_{10} \Box K_2$ after an attack on the vertices $v_4, v_5, u_0, u_1$. The graph on top corresponds to the initial configuration of the guards.}
	\label{Figure:C10-1}
\end{figure}


\newpage
\begin{figure}[h]
    \centering
	\begin{tikzpicture}[scale=.7]
		\begin{pgfonlayer}{nodelayer}
			\node [style=vertex, label=left:\color{blue}{\tiny $2/5$}] (0) at (-8.2, 11.2) {$v_0$};
			\node [style=vertex, label=above:\color{blue}{\tiny $1/5$}] (1) at (-7, 10) {$v_1$};
			\node [style=vertex, label=above:\color{blue}{\tiny $2/5$}] (2) at (-5, 10) {$v_2$};
			\node [style=vertex, label=above:\color{blue}{\tiny $1/5$}] (3) at (-3, 10) {$v_3$};
			\node [style=vertex, label=above:\color{blue}{\tiny $2/5$}] (4) at (-1, 10) {$v_4$};
			\node [style=vertex, label=above:\color{blue}{\tiny $1/5$}] (5) at (1, 10) {$v_5$};
			\node [style=vertex, label=above:\color{blue}{\tiny $1/5$}] (6) at (3, 10) {$v_6$};
			\node [style=vertex, label=above:\color{blue}{\tiny $2/5$}] (7) at (5, 10) {$v_7$};
			\node [style=vertex, label=above:\color{blue}{\tiny $1/5$}] (8) at (7, 10) {$v_8$};
			\node [style=vertex, label=right:\color{blue}{\tiny $1/5$}] (9) at (8.2, 11.2) {$v_9$};
			\node [style=vertex, label=left:\color{blue}{\tiny $1/5$}] (10) at (-8.2, 6.8) {$u_0$};
			\node [style=vertex, label=below:\color{blue}{\tiny $0$}] (11) at (-7, 8) {$u_1$};
			\node [style=vertex, label=below:\color{blue}{\tiny $1$}] (12) at (-5, 8) {$u_2$};
			\node [style=vertex, label=below:\color{blue}{\tiny $0$}] (13) at (-3, 8) {$u_3$};
			\node [style=vertex, label=below:\color{blue}{\tiny $1/5$}] (14) at (-1, 8) {$u_4$};
			\node [style=vertex, label=below:\color{blue}{\tiny $2/5$}] (15) at (1, 8) {$u_5$};
			\node [style=vertex, label=below:\color{blue}{\tiny $1/5$}] (16) at (3, 8) {$u_6$};
			\node [style=vertex, label=below:\color{blue}{\tiny $1/5$}] (17) at (5, 8) {$u_7$};
			\node [style=vertex, label=below:\color{blue}{\tiny $1/5$}] (18) at (7, 8) {$u_8$};
			\node [style=vertex, label=right:\color{blue}{\tiny $2/5$}] (19) at (8.2, 6.8) {$u_9$};
			
			\node [style=vertex, label=left:\color{blue}{\tiny $1/5$}] (20) at (-8.2, 5.2) {$v_0$};
			\node [style=vertex, label=above:\color{blue}{\tiny $2/5$}] (21) at (-7, 4) {$v_1$};
			\node [style=vertex, label=above:\color{blue}{\tiny $1/5$}] (22) at (-5, 4) {$v_2$};
			\node [style=vertex, label=above:\color{blue}{\tiny $2/5$}] (23) at (-3, 4) {$v_3$};
			\node [style=vertex, label=above:\color{blue}{\tiny $1/5$}] (24) at (-1, 4) {$v_4$};
			\node [style=vertex, label=above:\color{blue}{\tiny $2/5$}] (25) at (1, 4) {$v_5$};
			\node [style=vertex, label=above:\color{blue}{\tiny $1/5$}] (26) at (3, 4) {$v_6$};
			\node [style=vertex, label=above:\color{blue}{\tiny $1/5$}] (27) at (5, 4) {$v_7$};
			\node [style=vertex, label=above:\color{blue}{\tiny $2/5$}] (28) at (7, 4) {$v_8$};
			\node [style=vertex, label=right:\color{blue}{\tiny $1/5$}] (29) at (8.2, 5.2) {$v_9$};
			\node [style=vertex, label=left:\color{blue}{\tiny $2/5$}] (30) at (-8.2, 0.8) {$u_0$};
			\node [style=vertex, label=below:\color{blue}{\tiny $1/5$}] (31) at (-7, 2) {$u_1$};
			\node [style=vertex, label=below:\color{blue}{\tiny $0$}] (32) at (-5, 2) {$u_2$};
			\node [style=vertex, label=below:\color{blue}{\tiny $1$}] (33) at (-3, 2) {$u_3$};
			\node [style=vertex, label=below:\color{blue}{\tiny $0$}] (34) at (-1, 2) {$u_4$};
			\node [style=vertex, label=below:\color{blue}{\tiny $1/5$}] (35) at (1, 2) {$u_5$};
			\node [style=vertex, label=below:\color{blue}{\tiny $2/5$}] (36) at (3, 2) {$u_6$};
			\node [style=vertex, label=below:\color{blue}{\tiny $1/5$}] (37) at (5, 2) {$u_7$};
			\node [style=vertex, label=below:\color{blue}{\tiny $1/5$}] (38) at (7, 2) {$u_8$};
			\node [style=vertex, label=right:\color{blue}{\tiny $1/5$}] (39) at (8.2, 0.8) {$u_9$};
			
			\node [style=vertex, label=left:\color{blue}{\tiny $1/5$}] (40) at (-8.2, -0.8) {$v_0$};
			\node [style=vertex, label=above:\color{blue}{\tiny $1/5$}] (41) at (-7, -2) {$v_1$};
			\node [style=vertex, label=above:\color{blue}{\tiny $2/5$}] (42) at (-5, -2) {$v_2$};
			\node [style=vertex, label=above:\color{blue}{\tiny $1/5$}] (43) at (-3, -2) {$v_3$};
			\node [style=vertex, label=above:\color{blue}{\tiny $2/5$}] (44) at (-1, -2) {$v_4$};
			\node [style=vertex, label=above:\color{blue}{\tiny $1/5$}] (45) at (1, -2) {$v_5$};
			\node [style=vertex, label=above:\color{blue}{\tiny $2/5$}] (46) at (3, -2) {$v_6$};
			\node [style=vertex, label=above:\color{blue}{\tiny $1/5$}] (47) at (5, -2) {$v_7$};
			\node [style=vertex, label=above:\color{blue}{\tiny $1/5$}] (48) at (7, -2) {$v_8$};
			\node [style=vertex, label=right:\color{blue}{\tiny $2/5$}] (49) at (8.2, -0.8) {$v_9$};
			\node [style=vertex, label=left:\color{blue}{\tiny $1/5$}] (50) at (-8.2, -5.2) {$u_0$};
			\node [style=vertex, label=below:\color{blue}{\tiny $2/5$}] (51) at (-7, -4) {$u_1$};
			\node [style=vertex, label=below:\color{blue}{\tiny $1/5$}] (52) at (-5, -4) {$u_2$};
			\node [style=vertex, label=below:\color{blue}{\tiny $0$}] (53) at (-3, -4) {$u_3$};
			\node [style=vertex, label=below:\color{blue}{\tiny $1$}] (54) at (-1, -4) {$u_4$};
			\node [style=vertex, label=below:\color{blue}{\tiny $0$}] (55) at (1, -4) {$u_5$};
			\node [style=vertex, label=below:\color{blue}{\tiny $1/5$}] (56) at (3, -4) {$u_6$};
			\node [style=vertex, label=below:\color{blue}{\tiny $2/5$}] (57) at (5, -4) {$u_7$};
			\node [style=vertex, label=below:\color{blue}{\tiny $1/5$}] (58) at (7, -4) {$u_8$};
			\node [style=vertex, label=right:\color{blue}{\tiny $1/5$}] (59) at (8.2, -5.2) {$u_9$};

			\node [style=vertex, label=left:\color{blue}{\tiny $2/5$}] (60) at (-8.2, -6.8) {$v_0$};
			\node [style=vertex, label=above:\color{blue}{\tiny $1/5$}] (61) at (-7, -8) {$v_1$};
			\node [style=vertex, label=above:\color{blue}{\tiny $1/5$}] (62) at (-5, -8) {$v_2$};
			\node [style=vertex, label=above:\color{blue}{\tiny $2/5$}] (63) at (-3, -8) {$v_3$};
			\node [style=vertex, label=above:\color{blue}{\tiny $1/5$}] (64) at (-1, -8) {$v_4$};
			\node [style=vertex, label=above:\color{blue}{\tiny $2/5$}] (65) at (1, -8) {$v_5$};
			\node [style=vertex, label=above:\color{blue}{\tiny $1/5$}] (66) at (3, -8) {$v_6$};
			\node [style=vertex, label=above:\color{blue}{\tiny $2/5$}] (67) at (5, -8) {$v_7$};
			\node [style=vertex, label=above:\color{blue}{\tiny $1/5$}] (68) at (7, -8) {$v_8$};
			\node [style=vertex, label=right:\color{blue}{\tiny $1/5$}] (69) at (8.2, -6.8) {$v_9$};
			\node [style=vertex, label=left:\color{blue}{\tiny $1/5$}] (70) at (-8.2, -11.2) {$u_0$};
			\node [style=vertex, label=below:\color{blue}{\tiny $1/5$}] (71) at (-7, -10) {$u_1$};
			\node [style=vertex, label=below:\color{blue}{\tiny $2/5$}] (72) at (-5, -10) {$u_2$};
			\node [style=vertex, label=below:\color{blue}{\tiny $1/5$}] (73) at (-3, -10) {$u_3$};
			\node [style=vertex, label=below:\color{blue}{\tiny $0$}] (74) at (-1, -10) {$u_4$};
			\node [style=vertex, label=below:\color{blue}{\tiny $1$}] (75) at (1, -10) {$u_5$};
			\node [style=vertex, label=below:\color{blue}{\tiny $0$}] (76) at (3, -10) {$u_6$};
			\node [style=vertex, label=below:\color{blue}{\tiny $1/5$}] (77) at (5, -10) {$u_7$};
			\node [style=vertex, label=below:\color{blue}{\tiny $2/5$}] (78) at (7, -10) {$u_8$};
			\node [style=vertex, label=right:\color{blue}{\tiny $1/5$}] (79) at (8.2, -11.2) {$u_9$};
		\end{pgfonlayer}
		\begin{pgfonlayer}{edgelayer}
			\draw (0) to (1);
			\draw (1) to (2);
			\draw (2) to (3);
			\draw (3) to (4);
			\draw (4) to (5);
			\draw (5) to (6);
			\draw (6) to (7);
			\draw (7) to (8);
			\draw (8) to (9);
			\draw (9) to (0);
			\draw (10) to (11);
			\draw (11) to (12);
			\draw (12) to (13);
			\draw (13) to (14);
			\draw (14) to (15);
			\draw (15) to (16);
			\draw (16) to (17);
			\draw (17) to (18);
			\draw (18) to (19);
			\draw (19) to (10);
			\draw (0) to (10);
			\draw (1) to (11);
			\draw (2) to (12);
			\draw (3) to (13);
			\draw (4) to (14);
			\draw (5) to (15);
			\draw (6) to (16);
			\draw (7) to (17);
			\draw (8) to (18);
			\draw (9) to (19);
			\draw (20) to (21);
			\draw (21) to (22);
			\draw (22) to (23);
			\draw (23) to (24);
			\draw (24) to (25);
			\draw (25) to (26);
			\draw (26) to (27);
			\draw (27) to (28);
			\draw (28) to (29);
			\draw (29) to (20);
			\draw (30) to (31);
			\draw (31) to (32);
			\draw (32) to (33);
			\draw (33) to (34);
			\draw (34) to (35);
			\draw (35) to (36);
			\draw (36) to (37);
			\draw (37) to (38);
			\draw (38) to (39);
			\draw (39) to (30);
			\draw (20) to (30);
			\draw (21) to (31);
			\draw (22) to (32);
			\draw (23) to (33);
			\draw (24) to (34);
			\draw (25) to (35);
			\draw (26) to (36);
			\draw (27) to (37);
			\draw (28) to (38);
			\draw (29) to (39);
			\draw (40) to (41);
			\draw (41) to (42);
			\draw (42) to (43);
			\draw (43) to (44);
			\draw (44) to (45);
			\draw (45) to (46);
			\draw (46) to (47);
			\draw (47) to (48);
			\draw (48) to (49);
			\draw (49) to (40);
			\draw (50) to (51);
			\draw (51) to (52);
			\draw (52) to (53);
			\draw (53) to (54);
			\draw (54) to (55);
			\draw (55) to (56);
			\draw (56) to (57);
			\draw (57) to (58);
			\draw (58) to (59);
			\draw (59) to (50);
			\draw (40) to (50);
			\draw (41) to (51);
			\draw (42) to (52);
			\draw (43) to (53);
			\draw (44) to (54);
			\draw (45) to (55);
			\draw (46) to (56);
			\draw (47) to (57);
			\draw (48) to (58);
			\draw (49) to (59);
			\draw (60) to (61);
			\draw (61) to (62);
			\draw (62) to (63);
			\draw (63) to (64);
			\draw (64) to (65);
			\draw (65) to (66);
			\draw (66) to (67);
			\draw (67) to (68);
			\draw (68) to (69);
			\draw (69) to (60);
			\draw (70) to (71);
			\draw (71) to (72);
			\draw (72) to (73);
			\draw (73) to (74);
			\draw (74) to (75);
			\draw (75) to (76);
			\draw (76) to (77);
			\draw (77) to (78);
			\draw (78) to (79);
			\draw (79) to (70);
			\draw (60) to (70);
			\draw (61) to (71);
			\draw (62) to (72);
			\draw (63) to (73);
			\draw (64) to (74);
			\draw (65) to (75);
			\draw (66) to (76);
			\draw (67) to (77);
			\draw (68) to (78);
			\draw (69) to (79);
		\end{pgfonlayer}
	\end{tikzpicture}
	
    \caption{Configuration of the guards in $C_{10} \Box K_2$ after an attack on the vertices $u_2, u_3, u_4, u_5$. The graph on top corresponds to the initial configuration of the guards.}
	\label{Figure:C10-2}
\end{figure}


\begin{table}[h]
\begin{center}
    \caption{Response of the guards to each attack on the vertices of $C_{10} \Box K_2$.}
	\label{Table:C10}
	\begin{tabular}{ || m{1cm} || m{14cm} || } 
		\hline
		\small Attack & \small Response \\ 
		\hline\hline
		$v_1$	&
			\small
			$v_0 \xrightarrow[]{1} v_1$,
			$v_2 \xrightarrow[]{1/5} v_3$,
			$v_3 \xrightarrow[]{2/5} v_4$,
			$v_4 \xrightarrow[]{1/5} v_5$,
			$v_5 \xrightarrow[]{1/5} v_6$,
			$v_6 \xrightarrow[]{1/5} v_7$,
			$v_7 \xrightarrow[]{2/5} v_8$,
			$v_8 \xrightarrow[]{1/5} v_9$,
			$u_0 \xrightarrow[]{2/5} u_1$,
			$u_1 \xrightarrow[]{1/5} u_2$,
			$u_2 \xrightarrow[]{2/5} u_3$,
			$u_3 \xrightarrow[]{1/5} u_4$,
			$u_4 \xrightarrow[]{1/5} u_5$,
			$u_5 \xrightarrow[]{2/5} u_6$,
			$u_6 \xrightarrow[]{1/5} u_7$,
			$u_7 \xrightarrow[]{1/5} u_8$,
			$u_8 \xrightarrow[]{2/5} u_9$,
			$u_9 \xrightarrow[]{1/5} u_0$
			\\ \hline\hline
		$v_2$	&
			\small
			$v_0 \xrightarrow[]{2/5} v_9$,
			$v_0 \xrightarrow[]{2/5} u_0$,
			$v_3 \xrightarrow[]{2/5} v_2$,
			$v_6 \xrightarrow[]{1/5} v_5$,
			$v_7 \xrightarrow[]{1/5} v_6$,
			$u_0 \xrightarrow[]{1/5} u_1$,
			$u_0 \xrightarrow[]{1/5} u_9$,
			$u_1 \xrightarrow[]{1/5} u_2$,
			$u_2 \xrightarrow[]{2/5} v_2$,
			$u_3 \xrightarrow[]{1/5} u_2$,
			$u_4 \xrightarrow[]{1/5} u_3$,
			$u_5 \xrightarrow[]{2/5} u_4$,
			$u_6 \xrightarrow[]{1/5} u_5$,
			$u_7 \xrightarrow[]{1/5} u_6$,
			$u_8 \xrightarrow[]{2/5} u_7$,
			$u_9 \xrightarrow[]{1/5} u_8$
			\\ \hline\hline
		$v_3$	&
			\small
			$v_0 \xrightarrow[]{1/5} v_1$,
			$v_0 \xrightarrow[]{1/5} v_9$,
			$v_0 \xrightarrow[]{1/5} u_0$,
			$v_2 \xrightarrow[]{1/5} v_3$,
			$v_4 \xrightarrow[]{1/5} v_3$,
			$v_7 \xrightarrow[]{1/5} v_6$,
			$u_0 \xrightarrow[]{1/5} u_1$,
			$u_0 \xrightarrow[]{1/5} u_9$,
			$u_2 \xrightarrow[]{1/5} u_3$,
			$u_4 \xrightarrow[]{1/5} u_3$,
			$u_5 \xrightarrow[]{1/5} u_4$,
			$u_6 \xrightarrow[]{1/5} u_5$,
			$u_7 \xrightarrow[]{1/5} u_6$,
			$u_8 \xrightarrow[]{1/5} u_7$,
			$u_9 \xrightarrow[]{1/5} u_8$
			\\ \hline\hline
		$v_4$	&
			\small
			$v_0 \xrightarrow[]{1/5} v_1$,
			$v_0 \xrightarrow[]{1/5} v_9$,
			$v_0 \xrightarrow[]{1/5} u_0$,
			$v_3 \xrightarrow[]{2/5} v_4$,
			$v_5 \xrightarrow[]{1/5} v_4$,
			$u_0 \xrightarrow[]{1/5} u_1$,
			$u_0 \xrightarrow[]{1/5} u_9$,
			$u_1 \xrightarrow[]{1/5} u_2$,
			$u_2 \xrightarrow[]{1/5} u_3$,
			$u_3 \xrightarrow[]{1/5} u_4$,
			$u_5 \xrightarrow[]{1/5} u_4$,
			$u_7 \xrightarrow[]{1/5} u_6$,
			$u_8 \xrightarrow[]{1/5} u_7$
			\\ \hline\hline
		$v_5$	&
			\small
			$v_0 \xrightarrow[]{1/5} v_1$,
			$v_0 \xrightarrow[]{1/5} v_9$,
			$v_0 \xrightarrow[]{2/5} u_0$,
			$v_3 \xrightarrow[]{2/5} v_2$,
			$v_4 \xrightarrow[]{1/5} v_5$,
			$v_6 \xrightarrow[]{1/5} v_5$,
			$v_7 \xrightarrow[]{2/5} v_8$,
			$u_0 \xrightarrow[]{1/5} u_1$,
			$u_0 \xrightarrow[]{1/5} u_9$,
			$u_1 \xrightarrow[]{1/5} u_2$,
			$u_2 \xrightarrow[]{2/5} u_3$,
			$u_3 \xrightarrow[]{1/5} u_4$,
			$u_4 \xrightarrow[]{1/5} u_5$
			$u_6 \xrightarrow[]{1/5} u_5$,
			$u_7 \xrightarrow[]{1/5} u_6$,
			$u_8 \xrightarrow[]{2/5} u_7$,
			$u_9 \xrightarrow[]{1/5} u_8$
			\\ \hline\hline
		$u_0$	& 
			\small
			$v_0 \xrightarrow[]{1/5} v_1$,
			$v_0 \xrightarrow[]{1/5} v_9$,
			$v_0 \xrightarrow[]{1/5} u_0$,
			$v_3 \xrightarrow[]{1/5} v_2$,
			$v_7 \xrightarrow[]{1/5} v_8$,
			$u_1 \xrightarrow[]{1/5} u_0$,
			$u_2 \xrightarrow[]{1/5} u_3$,
			$u_5 \xrightarrow[]{1/5} v_5$,
			$u_8 \xrightarrow[]{1/5} u_7$,
			$u_9 \xrightarrow[]{1/5} u_0$
			\\ \hline\hline
		$u_1$	& 
			\small 
			$v_0 \xrightarrow[]{2/5} v_1$,
			$v_0 \xrightarrow[]{2/5} v_9$,
			$v_7 \xrightarrow[]{1/5} v_6$,
			$u_0 \xrightarrow[]{2/5} u_1$,
			$u_2 \xrightarrow[]{2/5} u_1$,
			$u_5 \xrightarrow[]{1/5} u_4$
			\\ \hline\hline
		$u_2$	& 
			\small 
			$v_0 \xrightarrow[]{1/5} v_1$,
			$v_0 \xrightarrow[]{1/5} v_9$,
			$v_0 \xrightarrow[]{1/5} u_0$,
			$v_2 \xrightarrow[]{1/5} u_2$,
			$v_3 \xrightarrow[]{2/5} v_2$,
			$v_4 \xrightarrow[]{1/5} v_3$,
			$v_5 \xrightarrow[]{1/5} v_4$,
			$v_6 \xrightarrow[]{1/5} v_5$,
			$v_7 \xrightarrow[]{1/5} v_6$,
			$v_8 \xrightarrow[]{1/5} v_7$,
			$u_0 \xrightarrow[]{2/5} u_9$,
			$u_1 \xrightarrow[]{1/5} u_2$,
			$u_3 \xrightarrow[]{1/5} u_2$,
			$u_4 \xrightarrow[]{1/5} v_4$,
			$u_5 \xrightarrow[]{1/5} u_4$,
			$u_6 \xrightarrow[]{1/5} u_5$,
			$u_7 \xrightarrow[]{1/5} u_6$,
			$u_8 \xrightarrow[]{1/5} u_7$,
			$u_8 \xrightarrow[]{1/5} v_8$,
			$u_9 \xrightarrow[]{1/5} u_8$	
		\\ \hline\hline
			$u_3$	& 
			\small 
			$v_0 \xrightarrow[]{2/5} v_1$,
			$v_0 \xrightarrow[]{1/5} v_9$,
			$v_0 \xrightarrow[]{1/5} u_0$,
			$v_3 \xrightarrow[]{1/5} u_3$,
			$v_4 \xrightarrow[]{1/5} v_3$,
			$v_5 \xrightarrow[]{1/5} v_4$,
			$v_6 \xrightarrow[]{1/5} v_5$,
			$v_7 \xrightarrow[]{1/5} v_6$,
			$u_0 \xrightarrow[]{1/5} u_9$,
			$u_2 \xrightarrow[]{2/5} u_3$,
			$u_3 \xrightarrow[]{1/5} u_3$,
			$u_5 \xrightarrow[]{1/5} v_5$,
			$u_7 \xrightarrow[]{1/5} u_6$,
			$u_8 \xrightarrow[]{1/5} u_7$,
			$u_9 \xrightarrow[]{1/5} u_8$,
			$u_8 \xrightarrow[]{1/5} v_8$
		\\ \hline\hline
			$u_4$	& 
			\small 
			$v_0 \xrightarrow[]{1/5} v_1$,
			$v_0 \xrightarrow[]{2/5} v_9$,
			$v_0 \xrightarrow[]{1/5} u_0$,
			$v_3 \xrightarrow[]{1/5} v_4$,
			$v_4 \xrightarrow[]{1/5} u_4$,
			$v_5 \xrightarrow[]{1/5} v_4$,	
			$v_6 \xrightarrow[]{1/5} v_5$,
			$v_7 \xrightarrow[]{2/5} v_6$,
			$u_0 \xrightarrow[]{1/5} u_1$,
			$u_0 \xrightarrow[]{1/5} u_9$,
			$u_3 \xrightarrow[]{1/5} u_4$,
			$u_5 \xrightarrow[]{1/5} u_4$,
			$u_8 \xrightarrow[]{2/5} u_7$,
			$u_9 \xrightarrow[]{1/5} u_8$
			\\ \hline\hline
		$u_5$	& 
			\small
			$v_0 \xrightarrow[]{1/5} v_1$,
			$v_0 \xrightarrow[]{1/5} v_9$,
			$v_0 \xrightarrow[]{1/5} u_0$,
			$v_2 \xrightarrow[]{1/5} v_3$,
			$v_3 \xrightarrow[]{1/5} v_4$,
			$v_4 \xrightarrow[]{1/5} v_5$,	
			$v_5 \xrightarrow[]{1/5} u_5$,
			$v_6 \xrightarrow[]{1/5} v_5$,
			$v_7 \xrightarrow[]{1/5} v_6$,
			$v_8 \xrightarrow[]{1/5} v_7$,	
			$u_0 \xrightarrow[]{1/5} u_1$,
			$u_0 \xrightarrow[]{1/5} u_9$,
			$u_1 \xrightarrow[]{1/5} u_2$,
			$u_2 \xrightarrow[]{1/5} v_2$,
			$u_4 \xrightarrow[]{1/5} u_5$,
			$u_6 \xrightarrow[]{1/5} u_5$,
			$u_8 \xrightarrow[]{1/5} v_8$,
			$u_9 \xrightarrow[]{1/5} u_8$
			\\ \hline\hline
	\end{tabular}
\end{center}
\end{table}

\end{document}